\documentclass[12pt]{article}

\usepackage[T1]{fontenc}
\usepackage[utf8]{inputenc}
\usepackage[english]{babel}
\usepackage{courier}
\usepackage{graphicx}
\usepackage{float}
\usepackage{amsthm}
\usepackage{amsmath}
\usepackage{amssymb}
\usepackage{concrete}
\usepackage{geometry}
\usepackage{caption}
\usepackage{subcaption}
\usepackage{algorithm}
\usepackage{algorithmic}
\usepackage{tikz}
\usepackage{wrapfig}
\usepackage{multirow}
\usepackage{hyperref}
\usepackage{sidecap}
\usepackage{authblk}

\theoremstyle{plain}
\newtheorem{theorem}{Theorem}
\newtheorem{lemma}{Lemma}
\newtheorem{proposition}{Proposition}

\theoremstyle{definition}
\newtheorem{definition}{Definition}

\theoremstyle{remark}
\newtheorem{cons}{Corollary}

\newcommand{\rank}{\mathrm{rank}\kern 2pt}

\newcommand{\mathC}{\mathbb{C}}
\newcommand{\cV}{{\cal V}_2}
\newcommand{\cVr}{{\cal V}_2^r}

\floatname{algorithm}{Algorithm}

\renewcommand{\geq}{\geqslant}
\renewcommand{\leq}{\leqslant}
\renewcommand{\ref}{\href}

\newcommand*{\myfont}{\fontfamily{phv}\selectfont}

\begin{document}

\title{Rectangular maximum volume and projective volume search algorithms.}

\author[1,2,3]{Osinsky~A.I. \thanks{a.osinskiy@skoltech.ru}}

\affil[1]{Moscow Institute of Physics and Technology, Dolgoprudny, Institutsky per., 9, Russia}
\affil[2]{Institute of Numerical Mathematics RAS, Moscow, Gubkina str., 8, Russia}
\affil[3]{Skolkovo Institute of Science and Technology Bolshoy Boulevard 30, bld. 1 Moscow, Russia}

\begin{abstract}
New methods for finding submatrices of (locally) maximal volume and large projective volume are proposed and studied.
Detailed analysis is also carried out for existing methods.
The effectiveness of the new methods is shown in the construction of cross approximations, and estimates are also proved in the case of their application for the search for a strongly nondegenerate submatrix.
Much attention is also paid to the choice of the starting submatrix.
\end{abstract}

\noindent{\it Keywords:}
Low-rank approximations; Pseudoskeleton approximations; Maximum volume principle  

\maketitle

\section{Introduction}

In this paper, we study algorithms aimed at finding submatrices of large volume and projective volume. The volume and the projective ($r$-projective) volume for an arbitrary matrix $A \in \mathC^{M \times N}$ are defined as follows:
\[
  V(A) = vol(A) = \cV(A) = \prod\limits_{i = 1}^{\min (M,N)} {{\sigma _i}(A)},
\]
\[
  \cVr(A) = \prod\limits_{i = 1}^{r} {{\sigma _i}(A)},
\]
where $\sigma_i(A)$ are singular values of $A$ in descending order (see Notations subsection).

The submatrices with close to the maximum volume (or projective volume) are known to yield precise cross approximations \cite{Cheb0, Cheb1, me}. 
Especially we pay attention to the projective volume, which guarantees better cross approximation accuracy estimates. High accuracy is confirmed in the numerical experiments.

An essential feature of the cross approximations is the use of a small part of matrix elements. This requires efficient algorithms for large volume (or large projective volume) submatrix search.

Since the maximum volume submatrix search is an NP-complete problem, in practice much simpler computational algorithms are used.
One such algorithm is called {\myfont maxvol} \cite{maxvol} (see algorithm \ref{maxvol-alg}). 
It finds a square submatrix of locally maximal volume \cite{Pan} also called dominant \cite{maxvol}. We will use the same definition for the rectangular submatrices.

\begin{definition}[\cite{Pan}, \cite{maxvol}]
A submatrix $\hat A \in \mathC^{m \times n}$ at the intersection of the rows $R \in \mathC^{M \times n}$ and columns $C \in \mathC^{m \times N}$ of the matrix $A \in \mathC^{M \times N}$ is called dominant or said to have a locally maximum volume if the volume of $\hat A$ does not increase by swapping one of its rows with another row from $C$ or one of its columns with another column from $R$.
\end{definition}

The notion of locally maximum projective volume is defined the same way: the projective volume should not increase with any single swap of rows or columns.

In practice, dominant submatrices are often in ``good'' rows and columns for the cross approximation. Moreover, there are estimates on the norm of the pseudoinverse matrix guaranteeing that dominant submatrices are relatively well conditioned.

Our algorithm {\myfont maxvol-rect} (algorithm \ref{maxvolrect-alg}) finds a dominant rectangular submatrix and it produces the same results as {\myfont maxvol} when searching for a square submatrix. This is the main difference of our algorithm from rectangular greedy maximum volume search algorithm {\myfont maxvol2} \cite{maxvol2} (see also algorithm \ref{maxvol2-alg} in section \ref{alg-sec}) later renamed as {\myfont rect\_maxvol} \cite{Mich}, which adds rows to increase the 2-volume (thus the original name) and never swaps them. We will stick to the original name when discussing this algorithm.

It was already used to construct cross approximations based on large projective volume in \cite{me}, and we will also use it that way. Later, in \cite{Mich}, another approach was proposed, but it requires the knowledge of $SVD$ and leads to an additional factor of two in the estimates, so we will not use it.

We also derive a faster version of {\myfont maxvol2}, which uses Householder reflections for the updates (Algorithm \ref{hmaxvol2-alg} in the Appendix). We will refer to the new version as ``Householder-based {\myfont maxvol2}'' while saying ``original {\myfont maxvol2}'' when referring to the original update formulas.

In this paper, we also consider another generalization of the maxvol algorithm called {\myfont maxvol-proj} (algorithm \ref{maxvolproj-alg}), which searches for large projective volume submatrices.

The main idea of the proposed algorithm for large projective volume is to search for ``good'' rows and columns separately.
Since the sought-for $m \times n$ submatrix is expected to be ``good'' only in the sense of its largest $r$ singular values, it is possible to start with $r \times n$ and $m \times r$ submatrices of large volume.

A choice of the starting submatrix is studied in detail, which provides a restriction on the number of row or column exchanges in the algorithms.

In the current section, we introduce the necessary notation and give Lemmas concerning the volume, projective volume, and dominant submatrices.

In section \ref{alg-sec}, the algorithms are constructed and briefly described. Detailed versions can be found in the appendix.

In section \ref{exp-sec}, the numerical experiments on random matrices are carried out. The results confirm the hypothesis of high low-rank approximation accuracy in the Frobenius norm.

\subsection{Notations}

Before proceeding further, we formulate the basic notation concerning column and cross approximations.

$A$ usually denotes the original matrix.
Its size is $M \times N,$ unless otherwise specified.

$\hat A$ is used to denote a submatrix of $A$, whose rows and columns generate a $CGR$-approxima\-tion (also called $CUR$ \cite{bestCUR}).

\begin{definition}
An approximation of the matrix $A \in \mathC^{M \times N}$ is called $CGR$ if it is described as a product of some of the columns $C \in \mathC^{M \times n}$ and rows $R \in \mathC^{m \times N}$ of the matrix $A$ and an arbitrary matrix $G \in \mathC^{m \times n}$ called a generator:
\[
  A \approx CGR.
\]
\end{definition}

Matrices $U$ and $V$ everywhere further contain orthonormal rows or columns.
Although they can be rectangular, we shall also call them ``unitary''. 
They often arise from the singular value decomposition for a matrix $A \in \mathC^{M \times N}$ of rank $R$ as follows:
\[
\begin{gathered}
  A = U\Sigma V, \hfill \\
  U \in {\mathbb{C}^{M \times R}}, \hfill \\
  \Sigma = diag({\sigma _1}, \ldots ,{\sigma _R}),\quad {\sigma _1} \geqslant {\sigma _2} \geqslant  \ldots  \geqslant {\sigma _R}, \hfill \\
  V \in {\mathbb{C}^{R \times N}}. \hfill \\ 
\end{gathered} 
\]
Then $\hat U \in \mathC^{n \times R}$ and $\hat V \in \mathC^{R \times n} $ denote some submatrices corresponding to $n$ rows $R$ and columns $C$ of $A$. $\sigma_i = \sigma_i(A)$ are singular values of the matrix $A$ in the descending order.

The approximation rank is denoted by $r$.

For an arbitrary matrix $B$ of rank at least $r$, the matrices $B_r$ and $B_r^+$ (the latter is called $r$ -pseudoinverse of $B$) are defined in terms of the singular value decomposition $B = U \Sigma V$ as follows:
\begin{eqnarray}
B_r  & = & U \Sigma_r V,\quad {\sigma _i}\left( {\Sigma_r } \right) = \left\{ \begin{array}{ll}
  \sigma _i \left( \Sigma  \right), & \text{if} \; i \leqslant r , \\
  0, & \text{otherwise}. \hfill \\ 
\end{array}  \right. \nonumber \\
B_r^+  & = & {V^*} \Sigma_r^+ U^*,\quad {\sigma _i}\left( {\Sigma_r ^+} \right) = \left\{ \begin{array}{ll}
  \sigma _i^{-1}\left( \Sigma  \right), & \text{if} \; i \leqslant r , \\
  0, & \text{otherwise}. \hfill \\ 
\end{array}  \right. \nonumber
\end{eqnarray}
Thus $B_r$ is the best rank-$r$ approximation of $B$ with respect to the 2-norm and the Frobenius norm. If $r$ coincides with the rank of $B$, we can write $B^+$ (exact pseudoinverse) instead of $B_r^+$.
$r$-pseudoinversion is further mainly used for submatrices of a large projective volume to construct a pseudo-skeleton approximation.

The range of natural numbers $\left\{ {1,2,3,...,n} \right\}$ is denoted by $\overline{1,n}$. The equality $\forall i = \overline{1,n}$ means that $i$ is an arbitrary value from the range.

\subsection{Properties of extreme submatrices}

This subsection describes the main properties of submatrices with a locally maximal volume or projective volume, proved in \cite{me}. 
Naturally, properties of the projective volume are also valid for volume.

\begin{lemma}[\cite{me}]\label{first-lem}
For an arbitrary $r \leqslant \min(m, n)$, the submatrix $\hat A \in \mathC^{m \times n}$ of the matrix $A = \left[ \begin{array}{cc} \hat A & b\end{array}\right] \in \mathC^{m \times (n+1)}$ with the locally maximal $r$-projective volume in $A$ and the submatrix $\hat U \in \mathC^{m \times n}$ of the matrix $U \in \mathC^{m \times N}, UU^* = I$ with the locally maximal $r$-projective volume in $U$, the following properties hold:
\begin{enumerate}
\item
For the $r$-projective volume of $\hat A$
\[
\cVr(\hat A) \sqrt{1 + \|\hat A_r^+ b\|_2^2} \leqslant \cVr(A) \leqslant \cVr(\hat A) \sqrt{\frac{n+1}{n-r+1}}.
\]
This implies
\[
\| \hat A_r^+ b \|_2^2 \leqslant \frac{r}{n - r + 1}.
\]
\item
For the norms of the $r$-pseudoinverse
\[
\|\hat{U}_r^+\|_2 \leqslant \sqrt{1 + \frac{r(N-n)}{n-r+1}},
\]
\[
\|\hat{U}_r^+\|_F \leqslant \sqrt{r + \frac{r(N-n)}{n-r+1}},
\]
\end{enumerate}
\end{lemma}

The second property is especially important since the cross approximation accuracy is affected by the properties of some submatrices of unitary matrices.

The next theorem shows that the locally maximal volume of long rectangular submatrices is close to the globally maximal volume, which may be false for square submatrices.

\begin{theorem}\label{maxvol2-cons}
Let $\hat A \in \mathbb{C}^{n \times r}$ be a submatrix of $A \in \mathbb{C}^{N \times r}$ with the maximal (nonzero) volume among all the submatrices, which differ from  $\hat A$  in a single row. 
Let $A_M \in \mathbb{C}^{n \times r}$ be a submatrix with the maximal volume in $A$. Then
\begin{equation}\label{maxvol2-cons_res}
{\cV}\left( {{A_M}} \right) \leqslant {\cV}\left( {\hat A} \right) \cdot \left( \frac{n+1}{n-r+1} \right)^{r/2} .
\end{equation}
\end{theorem}
\begin{proof}
Without loss of generality let the rows of $\hat A$ coincide with the first $n$ rows of $A$.
In justifying the {\myfont Dominant-C} algorithm (algorithm \ref{dominantc-alg}) in section~\ref{alg-sec}, we show the following. 
If the $j$-th row of $\hat A$ is replaced by the $i$-th row of $A$ then the ratio of the new volume $V_{new}$ to the old one $V_{old} = \cV(\hat A)$ equals
\begin{equation}\label{vol_div_pre}
  V_{new} / V_{old} = |C_{i,j}|^2 + (1 + \| C_i \|_2^2)(1 - \|C_j\|_2^2),
\end{equation}
where $C = A \hat A^+ \in \mathbb{C}^{N \times n}$.

Every submatrix of $C$ now has rank $r$, so it is natural to search for the maximal $r$-projective volume submatrix in $C$. Moreover, it corresponds to the maximal volume submatrix in $A$. Let an arbitrary submatrix $\tilde C \in \mathC^{n \times n}$ of $C$ correspond to the submatrix $\tilde A \in \mathC^{n \times r}$ in $A$. With $QR$ factorization of $\hat A$: $\hat A = Q R \in \mathC^{n \times r}$, $R \in \mathC^{r \times r}$, we see that
\[
  \cVr(\tilde C) = \cVr(\tilde A \hat A^+) = \cVr(\tilde A R^{-1} Q^*) = \cVr(\tilde A R^{-1}) = \cV(\tilde A R^{-1}) = \cV(\tilde A) \cV(R^{-1}) = \cV(\tilde A) \cV(\hat A^+).
\]
So all volumes are just multiplied by $\cV(\hat A^+)$. We denote the submatrices in $C$, corresponding to $\hat A$ and $A_M$ as $\hat C$ and $C_M$.

So now we have some matrix $C$, which contains submatrix $\hat C$ with locally maximum projective volume and submatrix $C_M$ with the maximum projective volume. Our task will be to estimate $\|C_M\|_F$ in this matrix. We are not going to make any swaps of the rows in $A$ or $C$ and so will not change $C$ in any way.

First, we estimate the squared 2-norm of the $j$-th row of $C$, denoted by $l_j$. Let $x = \mathop{max} \limits_{i} l_i$. 
Since we start from the locally maximum volume submatrix, the right hand side in (\ref {vol_div_pre}) does not exceed 1, which means the second term $(1+\|C_i\|_2^2)(1 - \|C_j\|_2^2) = (1+l_i)(1-l_j)$ is also not greater than one. With substitution $l_i = x$ we find that
\[
  \forall j = \overline{1,n}: \; \;(1-l_j)(1+x) \leqslant 1
\]
and therefore
\begin{equation}\label{lj_min}
  \forall j = \overline{1,n}: \; \; l_j \geqslant \frac{x}{1+x}.
\end{equation}

Let $A_M$ contain $k > 0$ rows from $\hat A$. The sum of the squared lengths of all rows in $C$ corresponding to $\hat A$ (that is, rows of $\hat C$) equals $r$. On the other hand, the sum of squared 2-norms of the rest $n - k$ rows, according to (\ref{lj_min}) is at least $(n - k) \frac{x}{1 + x}$.
So the Frobenius norm of the shared rows is at most
\[
  r - (n-k) \frac{x}{1+x}.
\]

Then the squared Frobenius norm of $C_M$ is bounded by
\[
  \|C_M\|_F^2 \leqslant (n-k)x + (r - (n-k)\frac{x}{1+x}) = r - (n-k)\frac{x^2}{1+x}.
\]  

The maximum of this expression is achieved at $k = 1$.
Then
\[
  \| C_M \|_F^2 \leqslant r+\frac{x^2(n-1)}{1+x}.
\]

This expression grows with increasing $x$. Choosing the maximum according to item 1 of the Lemma \ref{first-lem} value $x = \frac{r}{n - r + 1}$, we get
\[
  \| C_M \|_F^2 \leqslant r+\frac{r^2(n-1)}{(n-r+1)(n+1)} \leqslant r\frac{n+1}{n-r+1}.
\]

In the case $k = 0$, using item 1 of the Lemma \ref{first-lem} again for $n$ rows of $C,$ we obtain
\[
  \| C_M \|_F^2 \leqslant n\frac{r}{n-r+1},
\]
which is lower than the previous estimate.

The maximum projective volume is achieved in case of equal $r$ singular values.
This means they are all equal to $\sqrt{\frac{n + 1}{n-r + 1}}$, which proves (\ref{maxvol2-cons_res}).
\end{proof}

\begin{cons}
For $n \geqslant \frac{r^2}{2} + r - 1$, the volume ratio of the maximum volume submatrix to any locally maximum volume submatrix is bounded by $(1+2/r)^{r/2} \leqslant e = const$.
\end{cons}

At last, we give one more significant property.
\begin{lemma}[\cite{LowerBounds}]\label{for-lem}
For an arbitrary matrix $A \in \mathbb{C}^{M \times N}$ and any $r \leqslant \min(M, N)$, the following identity holds
\begin{equation}\label{lb_sig-eq}
\sum\limits_{T,\;|T|=r} {\cal V}^2(A_T) 
= \sum\limits_{1 \leqslant t_1 < t_2 < \cdots < t_r \leqslant \min\left\{M, N\right\}} 
\sigma^2_{t_1} \sigma^2_{t_2} \cdots \sigma^2_{t_r} \geqslant \left( \cVr(A) \right)^2, 
\end{equation}
where the sum is taken over all sets of $r$ indices $T = \left \{t_1, \ldots, t_r \right \}$.
$A_T$ are the rows of the matrix $A$ with indices from $T,$ and $\sigma_1 \geqslant \sigma_2 \geqslant \cdots \geqslant \sigma_{\min \{M, N \}}$ are the singular values of the matrix $A.$
\end{lemma}

The inequality in (\ref{lb_sig-eq}) is obtained by taking into account only the first summand.

\section{Algorithms}\label{alg-sec}

The primary intention of this section is to obtain and to justify new algorithms. In particular, an algorithm for finding dominant rectangular submatrices is discussed.
We estimate the iteration numbers and the computational complexity of each algorithm.
The complete descriptions with update formulas are given in the Appendix.

All the basic algorithms use fast updates to replace one of the rows or columns of the current submatrix. The replacements are performed to increase the volume.
The criteria for such exchanges are given in the current section.

Figure \ref{scheme} shows the general scheme for large volume and projective volume submatrix search. All the algorithms are enumerated according to the sequence of their execution. In practice, of course, some of the steps (algorithms) can be omitted. However, additional steps can improve approximation quality and provide guarantees on the total complexity. These guarantees and estimates are discussed in the corresponding subsections. Also, algorithms on the bottom of the diagram are derived similarly to the simpler algorithms before them, so it is convenient to talk about simpler approaches first and then move on to more complex ones.

Let us briefly describe the general approach. We start from some $r$ nondegenerate columns (which can be chosen randomly), and we want to start from a highly nondegenerate $k \times r$ submatrix in these columns to reduce the number of steps for further algorithms. This is the purpose of {\myfont pre-maxvol} (algorithm \ref{premaxvol-alg}), described in subsection \ref{premaxvolsubsec}. Essentially, it just uses $QR$ with pivoting. Then we can find $r \times r$ locally maximum volume submatrix with {\myfont maxvol} \cite{maxvol} (algorithm \ref{maxvol-alg}), described in the subsection \ref{maxvolsubsec}.

If we want to find a locally maximum volume rectangular submatrix, then we, first of all, need to add some rows or columns using {\myfont maxvol2} \cite{maxvol2} (algorithm \ref{maxvol2-alg}), described in the subsection \ref{maxvol2subsec}. Note also, that we present an asymptotically faster analog in the Appendix.

Next we move on to the new algorithms.

To find locally maximum volume rectangular submatrix, we use {\myfont maxvol-rect} (algorithm \ref{maxvolrect-alg}), described in the subsection \ref{maxvolrectsubsec}. It uses different criteria for updates along smaller and larger size, described in {\myfont Dominant-C} (algorithm \ref{dominantc-alg}) and {\myfont Dominant-R} (algorithm \ref{dominantr-alg}). Note that {\myfont Dominant-R} uses the same criterion and similar updates to the strong rank revealing $QR$, described in \cite{rrqr}.

Finally, in the subsection \ref{maxvolprojsubsec}, rectangular locally maximum volume submatrices are used to find large projective volume submatrix of arbitrary size (see {\myfont maxvol-proj}, algorithm \ref{maxvolproj-alg}).

After all the algorithms are presented, we discuss possible improvements and simplifications in the subsection \ref{impsubsec} and make a comparison with other subset selection algorithms in the subsection \ref{appsubsec}.

\begin{figure}[H]
\center
\vspace{-1cm}
\begin{tikzpicture}

	\makeatletter
		\let\cs\@undefined
	\makeatother

	\newlength{\cs} 
	\setlength{\cs}{0.8cm}
	
    \node [align = center] at (1\cs, -11\cs) {\Large maxvol\\ \Large -rect\\ \large Alg. \ref{maxvolrect-alg}};	
    \node [align = center] at (-3.3\cs, -15.6\cs) {\large maxvol\\ \large -proj \\ \large Alg. \ref{maxvolproj-alg}};	
    \node [align = center] at (1\cs, -4\cs) {\Large maxvol2 \\ \large Alg. \ref{maxvol2-alg}};
    
    \node at (-7\cs, 5.3\cs) {\large Alg. \ref{premaxvol-alg}};
    \node at (2\cs, 5.3\cs) {\large Alg. \ref{maxvol-alg}};
    \node at (-6.8\cs, -8.6\cs) {\large Alg. \ref{dominantc-alg}};
    \node at (4.2\cs, -8.6\cs) {\large Alg. \ref{dominantr-alg}};
    
    \node [circle, thick, draw, minimum size = 30mm] at (1\cs, -4\cs) {};
    \node [circle, thick, draw, minimum size = 30mm] at (1\cs, -11\cs) {};
	
	\node at (-6.3\cs, 6\cs) {\large pre-maxvol };
	\node at (2.1\cs, 6\cs) {\large maxvol };
	\node at (-6\cs, -8\cs) {\large Dominant-C };
	\node at (5\cs, -8\cs) {\large Dominant-R };
	
	
	\draw [thick, rounded corners]	(-8\cs, 6.5\cs) rectangle +(9\cs, -7.5\cs);
	
	\foreach \x in {-4\cs}
		\foreach \y in {0.4\cs}
	{\Large
	
		\draw [thick] (\x, \y) rectangle +(2\cs, 5\cs);
		\draw [thick] (\x, \y + 4\cs) rectangle +(2\cs, 1\cs);

        \node at (\x + 1\cs, \y + 4.5\cs) {\( \downarrow \)};
		
		\node at (\x + 1\cs, \y + 5.7\cs) {\( r \)};
		\node at (\x - 0.7\cs, \y + 4.5\cs) {\( k \)};
		\node at (\x + 2.7\cs, \y + 2.5\cs) {\( M \)};
		
		\draw [|<->|, very thick] (\x, \y + 5.2\cs) -- (\x + 2\cs, \y + 5.2\cs);
		\draw [|<->|, very thick] (\x - 0.2\cs, \y + 4\cs) -- (\x - 0.2\cs, \y + 5\cs);
		
		\draw [|<->|, very thick] (\x + 2.2\cs, \y) -- (\x + 2.2\cs, \y + 5\cs);
	}
	
	
	\draw [thick, rounded corners]	(1.0\cs, 6.5\cs) rectangle +(9.5\cs, -7.5\cs);

	\foreach \x in {4.5\cs}
		\foreach \y in {0.4\cs}
	{\Large
		\draw [thick] (\x, \y) rectangle +(5\cs, 5\cs);
		\draw [thick] (\x + 1\cs, \y) rectangle +(1\cs, 5\cs);
		\draw [thick] (\x + 3\cs, \y) rectangle +(1\cs, 5\cs);
		\draw [thick] (\x, \y + 3.5\cs) rectangle +(5\cs, 1\cs);
		\draw [thick] (\x, \y + 1.5\cs) rectangle +(5\cs, 1\cs);

        \node at (\x + 3.5\cs, \y + 4\cs) {\( \hat A \)};
		\node at (\x + 1.5\cs, \y + 2\cs) {\( \hat A \)};
		\node at (\x + 3.5\cs, \y + 2\cs) {\( \hat A \)};
		\node at (\x + 3.5\cs, \y + 3\cs) {\( \downarrow \)};
		\node at (\x + 2.5\cs, \y + 2\cs) {\( \leftarrow \)};
		
		\node at (\x + 1.5\cs, \y + 5.7\cs) {\( r \)};
		\node at (\x + 5.7\cs, \y + 2\cs) {\( r \)};
		\node at (\x - 0.7\cs, \y + 2.5\cs) {\( M \)};
		\node at (\x + 2.5\cs, \y - 0.7\cs) {\( N \)};
		
		\draw [|<->|, very thick] (\x + 1\cs, \y + 5.2\cs) -- (\x + 2\cs, \y + 5.2\cs);
		\draw [|<->|, very thick] (\x + 5.2\cs, \y + 1.5\cs) -- (\x + 5.2\cs, \y + 2.5\cs);
		\draw [|<->|, very thick] (\x - 0.2\cs, \y) -- (\x - 0.2\cs, \y + 5\cs);
		\draw [|<->|, very thick] (\x, \y - 0.2\cs) -- (\x + 5\cs, \y - 0.2\cs);
	}
	
	\draw [thick, rounded corners]	(-8\cs, -1\cs) rectangle +(18.5\cs, -6.5\cs);
	
	
	\foreach \x in {-4\cs}
		\foreach \y in {-7.1\cs}
	{\Large
		\draw [thick] (\x, \y) rectangle +(1\cs, 5\cs);
		\draw [thick] (\x, \y + 3\cs) rectangle +(1\cs, 2\cs);
		\draw [thick] (\x, \y + 4\cs) rectangle +(1\cs, 1\cs);

        \node at (\x + 0.5\cs, \y + 4.5\cs) {\( \hat A \)};
        \node at (\x + 0.5\cs, \y + 3.5\cs) {\( \downarrow \)};
		
		\node at (\x + 0.5\cs, \y + 5.7\cs) {\( r \)};
		\node at (\x - 0.7\cs, \y + 4\cs) {\( n \)};
		\node at (\x + 1.7\cs, \y + 2.5\cs) {\( M \)};
		
		\draw [|<->|, very thick] (\x, \y + 5.2\cs) -- (\x + 1\cs, \y + 5.2\cs);
		\draw [|<->|, very thick] (\x - 0.2\cs, \y + 3\cs) -- (\x - 0.2\cs, \y + 5\cs);
		
		\draw [|<->|, very thick] (\x + 1.2\cs, \y) -- (\x + 1.2\cs, \y + 5\cs);
	}
	
	
	\foreach \x in {9.5\cs}
		\foreach \y in {-5.1\cs}
	{\Large
		\draw [thick] (\x, \y) rectangle +(-5\cs, 1\cs);
		\draw [thick] (\x - 3\cs, \y) rectangle +(-2\cs, 1\cs);
		\draw [thick] (\x - 4\cs, \y) rectangle +(-1\cs, 1\cs);

        \node at (\x - 4.5\cs, \y + 0.5\cs) {\( \hat A \)};
        \node at (\x - 3.5\cs, \y + 0.5\cs) {\( \rightarrow \)};
		
		\node at (\x - 5.7\cs, \y + 0.5\cs) {\( r \)};
		\node at (\x - 4\cs, \y - 0.7\cs) {\( n \)};
		\node at (\x - 2.5\cs, \y + 1.7\cs) {\( N \)};
		
		\draw [|<->|, very thick] (\x - 5.2\cs, \y) -- (\x - 5.2\cs, \y + 1\cs);
		\draw [|<->|, very thick] (\x - 3\cs, \y - 0.2\cs) -- (\x - 5\cs, \y - 0.2\cs);
		
		\draw [|<->|, very thick] (\x, \y + 1.2\cs) -- (\x - 5\cs, \y + 1.2\cs);
	}
	
	\draw [thick, rounded corners]	(-8\cs, -7.5\cs) rectangle +(18.5\cs, -7\cs);
	
	
	\foreach \x in {-4\cs}
		\foreach \y in {-14\cs}
	{\Large
		\draw [thick] (\x, \y) rectangle +(1\cs, 5\cs);
		\draw [thick] (\x, \y + 3.5\cs) rectangle +(1\cs, 1.5\cs);
		\draw [thick] (\x, \y + 1\cs) rectangle +(1\cs, 1.5\cs);

        \node at (\x + 0.5\cs, \y + 4.25\cs) {\( \hat A \)};
        \node at (\x + 0.5\cs, \y + 1.75\cs) {\( \hat A \)};
        \node at (\x + 0.5\cs, \y + 3\cs) {\( \downarrow \)};
		
		\node at (\x + 0.5\cs, \y + 5.7\cs) {\( r \)};
		\node at (\x - 0.7\cs, \y + 4.25\cs) {\( n \)};
		\node at (\x + 1.7\cs, \y + 2.5\cs) {\( M \)};
		
		\draw [|<->|, very thick] (\x, \y + 5.2\cs) -- (\x + 1\cs, \y + 5.2\cs);
		\draw [|<->|, very thick] (\x - 0.2\cs, \y + 3.5\cs) -- (\x - 0.2\cs, \y + 5\cs);
		
		\draw [|<->|, very thick] (\x + 1.2\cs, \y) -- (\x + 1.2\cs, \y + 5\cs);
	}
	
	
	\foreach \x in {4.5\cs}
		\foreach \y in {-12\cs}
	{\Large
		\draw [thick] (\x, \y) rectangle +(5\cs, 2\cs);
		\draw [thick] (\x, \y) rectangle +(1\cs, 2\cs);
		\draw [thick] (\x + 2\cs, \y) rectangle +(1\cs, 2\cs);

        \node at (\x + 0.5\cs, \y + 1\cs) {\( \hat A \)};
        \node at (\x + 1.5\cs, \y + 1\cs) {\( \rightarrow \)};
        \node at (\x + 2.5\cs, \y + 1\cs) {\( \hat A \)};
		
		\node at (\x - 0.7\cs, \y + 1\cs) {\( m \)};
		\node at (\x + 0.5\cs, \y + 2.7\cs) {\( r \)};
		\node at (\x + 2.5\cs, \y - 0.7\cs) {\( N \)};
		
		\draw [|<->|, very thick] (\x - 0.2\cs, \y) -- (\x - 0.2\cs, \y + 2\cs);
		\draw [|<->|, very thick] (\x, \y + 2.2\cs) -- (\x + 1\cs, \y + 2.2\cs);
		
		\draw [|<->|, very thick] (\x, \y - 0.2\cs) -- (\x + 5\cs, \y - 0.2\cs);
	}
	
	
	\draw [thick, rounded corners]	(-4.5\cs, -14.5\cs) rectangle +(11.5\cs, -8\cs);
	
	\foreach \x in {-1\cs}
		\foreach \y in {-21\cs}
	{\Large
		\draw [thick] (\x, \y) rectangle +(5\cs, 5\cs);
		\draw [thick] (\x, \y) rectangle +(1\cs, 5\cs);
		\draw [thick] (\x + 1.5\cs, \y) rectangle +(2\cs, 5\cs);
		\draw [thick] (\x, \y + 4\cs) rectangle +(5\cs, 1\cs);
		\draw [thick] (\x, \y + 1.5\cs) rectangle +(5\cs, 2\cs);

        \node at (\x + 2.5\cs, \y + 2.5\cs) {\huge $\hat A$};
		
		\node at (\x + 0.5\cs, \y + 5.7\cs) {\( r \)};
		\node at (\x - 0.7\cs, \y + 2.5\cs) {\( m \)};
		\node at (\x - 0.7\cs, \y + 4.5\cs) {\( r \)};
		\node at (\x + 2.5\cs, \y + 5.7\cs) {\( n \)};
		\node at (\x + 5.7\cs, \y + 2.5\cs) {\( M \)};
		\node at (\x + 2.5\cs, \y - 0.7\cs) {\( N \)};
		
		{\tiny
		\node [align = center] at (\x + 2.5\cs, \y + 4.5\cs) {\normalsize maxvol\\ \normalsize -rect};
		\node [rotate = 90, align = center] at (\x + 0.5\cs, \y + 2.5\cs) {\normalsize maxvol\\ \normalsize -rect};
		}
		
		\draw [|<->|, very thick] (\x, \y + 5.2\cs) -- (\x + 1\cs, \y + 5.2\cs);
		\draw [|<->|, very thick] (\x + 1.5\cs, \y + 5.2\cs) -- (\x + 3.5\cs, \y + 5.2\cs);
		\draw [|<->|, very thick] (\x - 0.2\cs, \y + 4\cs) -- (\x - 0.2\cs, \y + 5\cs);
		\draw [|<->|, very thick] (\x - 0.2\cs, \y + 1.5\cs) -- (\x - 0.2\cs, \y + 3.5\cs);
		
		\draw [|<->|, very thick] (\x + 5.2\cs, \y) -- (\x + 5.2\cs, \y + 5\cs);
		\draw [|<->|, very thick] (\x, \y - 0.2\cs) -- (\x + 5\cs, \y - 0.2\cs);
	}
	
\end{tikzpicture}
\caption{General maximum volume and projective volume search scheme. Algorithms are enumerated according to the execution order and appearance in the text. Arrows show the new (appended) rows or columns (pre-maxvol, maxvol2) or changes of the current submatrix $\hat A$ (maxvol, maxvol-rect).}\label{scheme}
\end{figure}
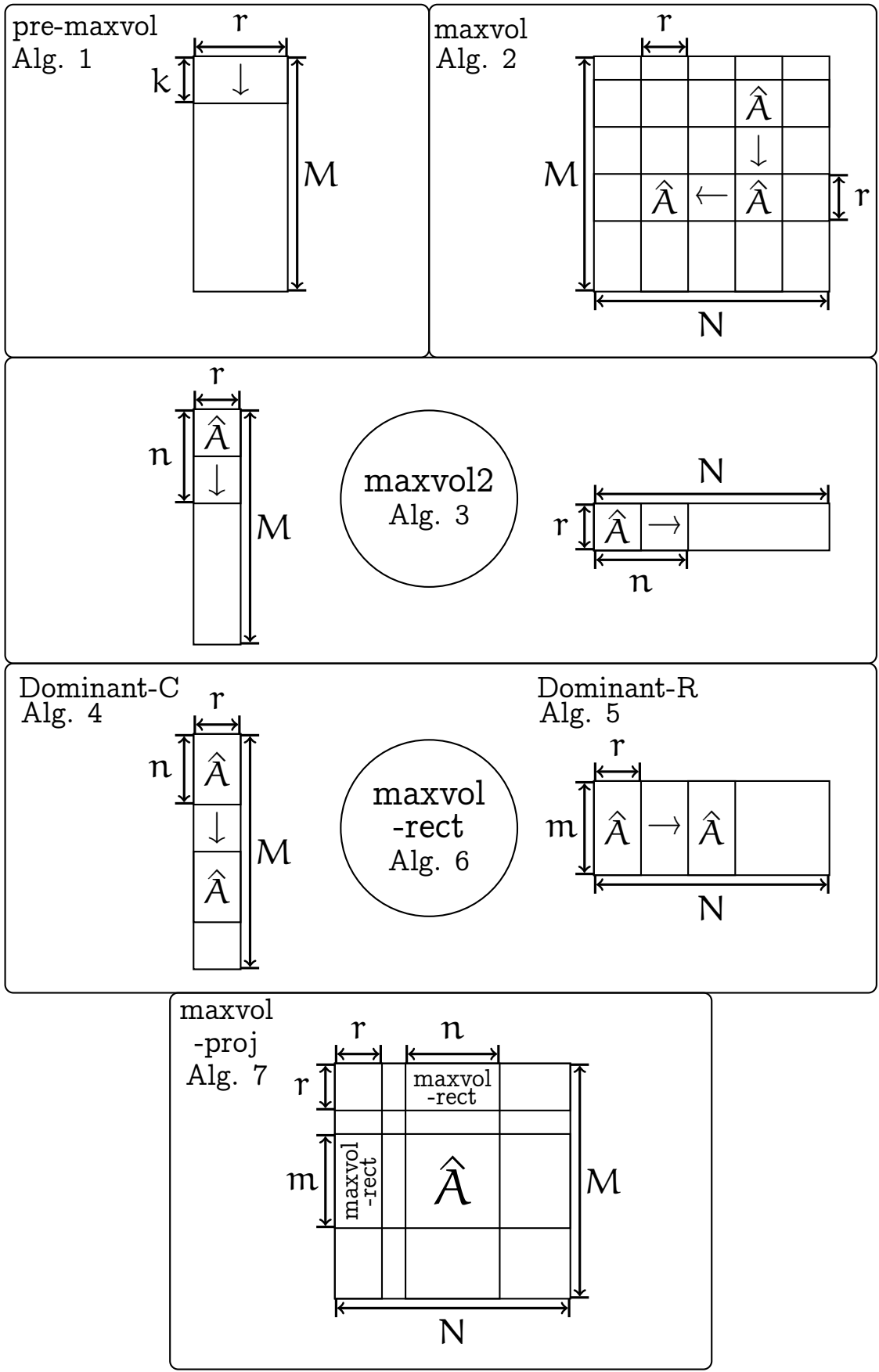

\subsection{How to start}\label{premaxvolsubsec}

\begin{wrapfigure}{R}{0.35\textwidth}
    \begin{tikzpicture}
	\makeatletter
		\let\cs\@undefined
	\makeatother

	\newlength{\cs} 
	\setlength{\cs}{1.0cm}

	\foreach \x in {1.5\cs}
		\foreach \y in {1\cs}
	{\huge
		\draw [thick] (\x, \y) rectangle +(2\cs, 5\cs);
		\draw [thick] (\x, \y + 4\cs) rectangle +(2\cs, 1\cs);

        \node at (\x + 1\cs, \y + 4.5\cs) {\( \downarrow \)};
		
		\node at (\x + 1\cs, \y + 5.7\cs) {\( n \)};
		\node at (\x - 0.7\cs, \y + 4.5\cs) {\( r \)};
		\node at (\x + 2.7\cs, \y + 2.5\cs) {\( M \)};
		
		\draw [|<->|, very thick] (\x, \y + 5.2\cs) -- (\x + 2\cs, \y + 5.2\cs);
		\draw [|<->|, very thick] (\x - 0.2\cs, \y + 4\cs) -- (\x - 0.2\cs, \y + 5\cs);
		
		\draw [|<->|, very thick] (\x + 2.2\cs, \y) -- (\x + 2.2\cs, \y + 5\cs);
	}
\end{tikzpicture}
    \caption{{\myfont pre-maxvol}: construction of $\hat A = A_{\mathcal{I},:}$. Pay attention to the notations for the number of rows and columns.}
    \label{premaxvol-fig}
\end{wrapfigure}
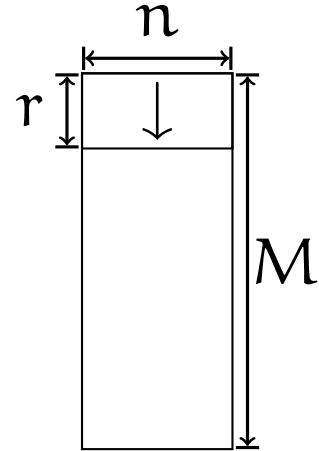

Our first task is to select some rows, corresponding to a submatrix with a sufficiently large volume so that we can decrease the number of steps for further algorithms (see figure \ref{premaxvol-fig}). To do it, we can construct $RRQR$ (rank-revealing $QR$) via Householder reflections with column pivoting.
This algorithm, called $QRP$, iteratively adds a new row/column corresponding to the row/column of the current error with the largest 2-norm (for a detailed description, see \cite{hqr}, page 278).
An alternative version of this algorithm, which we call {\myfont pre-maxvol}, is presented in the Appendix. {\myfont Pre-maxvol} is sufficient to obtain a submatrix with a volume that differs from the maximum volume submatrix by no more than $r!$ times (see proposition \ref{pre-prop}).

The ratio of the new volume to the old one is equal to the length of the new row in the orthogonal complement to the already selected rows.
These squared lengths will be stored in a vector $\gamma$.

\begin{algorithm}[h!]
\caption{{\myfont pre-maxvol} (The full version is a modification of the addition of rows in $RRQR$ from \cite{rrqr})}
\label{premaxvol-alg}
\begin{algorithmic}[1]
\REQUIRE{Matrix $A \in \mathbb{R}^{M \times n}$, required rank $r$.}
\ENSURE{Set of row indices $\mathcal{I}$ of cardinality $r$, containing a submatrix with volume no more than $r!$ less than the maximum.}
\STATE $\mathcal{I} := \emptyset$
\FOR{$i := 1$ \TO $M$}
  \STATE $\gamma_i := \| A_{i,:} \|_2^2$
\ENDFOR
\FOR{$k := 1$ \TO $r$}
  \STATE Select $j$, corresponding to $\mathop {\max }\limits_{j}  {\gamma_j}$
  \STATE $\mathcal{I} := \mathcal{I} \cup \{j\}$
  \STATE Update elements of $\gamma$: squared lengths of rows of $A$ after projection on the orthogonal complement to the rows in $\mathcal{I}$
\ENDFOR
\end{algorithmic}
\end{algorithm}

The complexity of {\myfont pre-maxvol} is $O(Mnr)$ for $A \in \mathbb{C}^{M \times n}$. 

The following statement was first proved in \cite{gr_maxvol}, where the greedy addition of rows was also investigated. Here we give a simpler proof using the projective volume.

\begin{proposition}[\cite{gr_maxvol}]\label{pre-prop}
The algorithm {\myfont pre-maxvol} finds a submatrix with the volume not more than $r!$ times smaller than the maximum.
\end{proposition}

\begin{proof}
We yield a proof by induction on $r$. The base $r = 1$ is obvious. Let us suppose that the proposition is true for an algorithm with $r-1$ rows and prove the induction step.

Let $a$ be the longest row in $A$. Then there is a submatrix in $\hat{A} \in \mathbb{C}^{r \times n}$, which contains $a$ and such that
\begin{equation}\label{prop1-eq}
\frac{{\cal V}(\hat{A})}{{\cal V} (A_M)} \geq \frac{1}{r},
\end{equation}
where $A_M \in \mathbb{C}^{r \times n}$ is the maximum volume submatrix. If $a$ is contained in $A_M$, then (\ref{prop1-eq}) is true for $\hat A = A_M$. Otherwise cnsider a submatrix $A' = \left[ {\begin{array}{*{20}{c}}
  a \\ 
  {{A_M}} 
\end{array}} \right] \in \mathC^{(r+1)^n}$, which is obtained by adding $a$ to $A_M$. From item 1 of Lemma 1 we have
\[
\frac{\left( \cVr(A') \right)^2}{\left( \cVr(A_M) \right)^2} \geqslant 1 + \|A_M^+ a \|_2^2,
\]
where we have pseudoinverse instead of $r$-pseudoinverse, because $\rank A_M = r$. On the other hand, from Lemma 2 
\[
  \sum\limits_{T,\;|T|=r} \left(\cV(A_T)\right)^2 \geqslant \left( \cVr(A') \right)^2,
\]
so
\[
\frac{\sum\limits_{T,\;|T|=r} \left(\cV(A_T)\right)^2}{\left( \cVr(A_M) \right)^2} \geqslant \frac{\left( \cVr(A') \right)^2}{\left( \cVr(A_M) \right)^2} \geqslant 1 + \|A_M^+ a \|_2^2.
\]
Next,
\[
  1 + \|A_M^+ c \|_2^2  \geqslant 1 + \frac{\| c \|_2^2}{\| A_M \|_2^2} \geqslant 1 + \frac{\| c \|_2^2}{\| A_M \|_F^2} \geqslant 1 + \frac{1}{r},
\]
since $c$ is not smaller than any row of $A_M$. Thus, the sum of the squared volumes of all the submatrices in $A'$ is not less than $1 + 1/r$. By subtracting the volume of $A_M$, the sum of squared volumes of the remaining $r$ submatrices is estimated as at least $1/r$. This fact implies that the squared volume of some submatrix is at least $1/r^2$ and proves (\ref{prop1-eq}).

It remains to use the induction hypothesis to see that the algorithm eventually finds a submatrix $A_a \in \mathbb{C}^{r \times n}$ containing $a$, such that
\begin{equation}\label{prop1-eq2}
\frac{{\cal V}(A_a)}{{\cal V}(\hat{A})} \geq \frac{1}{(r-1)!}.
\end{equation}
Indeed, let us follow the algorithm and change $A$ by replacing all the rows, starting from the second with their orthogonal complements to $a$. Any $r \times n$ submatrix containing $a$ has a volume, equal to $\|a\|_2$ times the volume of the remaining $r-1$ rows of this submatrix. $\frac{\cV(\hat A)}{\|a\|_2}$ is not greater than maximum volume in the orthogonal complement, so the volume of these $r-1$ rows is bounded from below by $\frac{\cV(\hat A)}{(r-1)!\|a\|_2}$ by the induction hypothesis, which proves (\ref{prop1-eq2}). Combining (\ref{prop1-eq}) and (\ref{prop1-eq2}) proves the propoition.
\end{proof}

\subsection{Square locally maximum volume search}\label{maxvolsubsec}

We move on to the {\myfont maxvol} algorithm (see figure \ref{maxvol-fig}). It selects the $r \times r$ submatrix of large volume.

The main idea of {\myfont maxvol} is to find a dominant submatrix. We remind that a submatrix is called dominant if replacing any single row (column) by another row (column) of the matrix does not increase its volume.
The criterion for submatrix dominance comes from the following Lemma.

\begin{lemma}[see proof of Lemma 1 in \cite{maxvol}]\label{maxvol-lem}
Let $\hat A \in \mathbb{C}^{r \times r}$ be a submatrix in the first $r$ rows of the matrix $A \in \mathC^{M \times r}$. Then, replacing the $j$-th row of the submatrix $\hat A$ with the $i$-th row of $A$ ($i > r$) changes the squared volume of $\hat A$ equal to $V_{old} = \cV(\hat A)$ as
\[
\begin{gathered}
 V_{new}^2 / V_{old}^2 = |C_{ij}|^2, \\
 C = A \hat A^{-1} \in \mathC^{M \times r}.
\end{gathered}
\]
\end{lemma}

It immediately follows that the $r \times r$ submatrix in fixed $r$ columns is dominant if and only if all absolute values of the elements in $C$ do not exceed 1, that is, $\| C \| _C \leqslant 1$. If this inequality is not satisfied, one can replace the rows.

\begin{algorithm}[h!]
\caption{{\myfont maxvol} \cite{maxvol}} \label{maxvol-alg}
\begin{algorithmic}[1]
\REQUIRE{Matrix $A \in \mathbb{C}^{M \times N}$, starting sets of row indices $\mathcal{I}$ and column indices $\mathcal{J}$ of cardinality $r$. For example, $\mathcal{I} = \mathcal{J} = \{1, ..., r \}$.}
\ENSURE{The row and column indices of the dominant submatrix of rank $r$ written in $\mathcal{I}$ and $\mathcal{J}$.
}
\WHILE{replacements occur}
  \FOR{$changes\_in$ \textbf{in} $\{\mathcal{I}, \mathcal{J}\}$} \label{alg_bit}
    \STATE $C := A_{:,\mathcal{J}} A_{\mathcal{I},\mathcal{J}}^{-1}$ if we change the selected set of rows $\mathcal{I}$
    \STATE $C := \left( A_{\mathcal{I},\mathcal{J}}^{-1} A_{\mathcal{I},:} \right)^*$ if we change the selected set of columns $\mathcal{J}$
    \STATE Select $i$ and $j$ corresponding to $\mathop {\max }\limits_{i,j} \left| {{C_{i,j}}} \right|$
    \WHILE{$\left| C_{i,j} \right| > 1$} \label{alg_sit}
      \STATE Update $C$
      \STATE Replace index $j$ with $i$ in the set $changes\_in$ (in either $\mathcal{I}$ or $\mathcal{J}$)
      \STATE Select $i$ and $j$ corresponding to $\mathop {\max }\limits_{i,j} \left| {{C_{i,j}}} \right|$ 
    \ENDWHILE
  \ENDFOR
\ENDWHILE
\end{algorithmic}
\end{algorithm}

\begin{wrapfigure}{R}{0.5\textwidth}
    \begin{tikzpicture}
	\makeatletter
		\let\cs\@undefined
	\makeatother

	\newlength{\cs} 
	\setlength{\cs}{1.0cm}

	\foreach \x in {1.5\cs}
		\foreach \y in {1\cs}
	{\huge
		\draw [thick] (\x, \y) rectangle +(5\cs, 5\cs);
		\draw [thick] (\x + 1\cs, \y) rectangle +(1\cs, 5\cs);
		\draw [thick] (\x + 3\cs, \y) rectangle +(1\cs, 5\cs);
		\draw [thick] (\x, \y + 3.5\cs) rectangle +(5\cs, 1\cs);
		\draw [thick] (\x, \y + 1.5\cs) rectangle +(5\cs, 1\cs);

        \node at (\x + 3.5\cs, \y + 4\cs) {\( \hat A \)};
		\node at (\x + 1.5\cs, \y + 2\cs) {\( \hat A \)};
		\node at (\x + 3.5\cs, \y + 2\cs) {\( \hat A \)};
		\node at (\x + 3.5\cs, \y + 3\cs) {\( \downarrow \)};
		\node at (\x + 2.5\cs, \y + 2\cs) {\( \leftarrow \)};
		
		\node at (\x + 1.5\cs, \y + 5.7\cs) {\( r \)};
		\node at (\x + 5.7\cs, \y + 2\cs) {\( r \)};
		\node at (\x - 0.7\cs, \y + 2.5\cs) {\( M \)};
		\node at (\x + 2.5\cs, \y - 0.7\cs) {\( N \)};
		
		\draw [|<->|, very thick] (\x + 1\cs, \y + 5.2\cs) -- (\x + 2\cs, \y + 5.2\cs);
		\draw [|<->|, very thick] (\x + 5.2\cs, \y + 1.5\cs) -- (\x + 5.2\cs, \y + 2.5\cs);
		\draw [|<->|, very thick] (\x - 0.2\cs, \y) -- (\x - 0.2\cs, \y + 5\cs);
		\draw [|<->|, very thick] (\x, \y - 0.2\cs) -- (\x + 5\cs, \y - 0.2\cs);
	}
\end{tikzpicture}
    \caption{Change of the current submatrix $\hat A = A_{\mathcal{I},\mathcal{J}}$ when the algorithm {\myfont maxvol} is executed.}
    \label{maxvol-fig}
\end{wrapfigure}
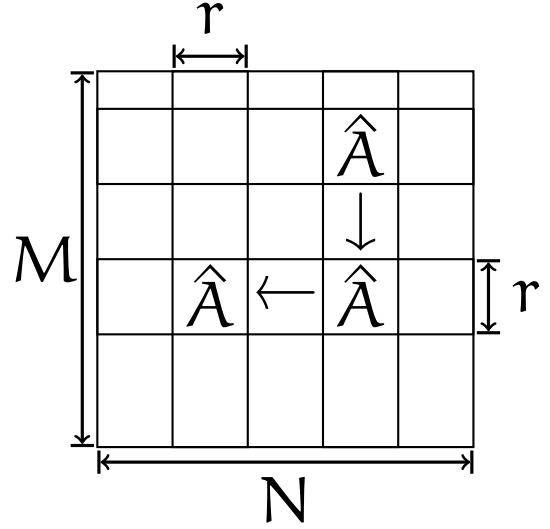

The algorithm complexity is $O((M+N)r\cdot iter + (M+N)r^2 \cdot IT)$ operations. Here $iter$ denotes the total number of replacements of rows and columns, that is, the number of iterations of ``while'' loop on the line \ref{alg_sit}; and $IT$ is the number of passes through rows and columns, that is, the total number of iterations of the ``for'' loop (line \ref{alg_bit}).

In practice $| C_{i, j} |$ is compared to $c = 1 + \varepsilon$ for small $\varepsilon$. This avoids instability and, as we will see below, allows to limit the number of row/column replacements.

So, let $|C_{i,j}|$ be compared with some constant $c > 1$ in the line \ref{alg_sit}. In this case, the algorithm no longer produces a dominant submatrix.
However, the volume of the resulting submatrix differs by no more than $c$ times from the volume of any submatrices that differs in a single row or column.

Let us denote by $\Gamma$ the ratio of the maximum volume to the volume of the starting submatrix.
Then for $\Gamma = (\alpha^2 r)^{r/2}$ we can guarantee the convergence to locally maximum volume submatrix within the corresponding rows or columns in
\[
  k = \left\lceil {\frac{{\ln \Gamma}}{{\ln c}}} \right\rceil - 1
\]
steps. Indeed, the $C$-norm of rows of the matrix $C$ cannot be greater than the ratio of the current volume to the maximum one, and this ratio for each iteration drops by at least $c$ times.

If $\alpha \leqslant c$, we need no more than
\[
  k_1 = \left\lceil r\frac{{\ln c^2 r}}{2\ln c} \right\rceil - 1 \leqslant \left\lceil r \left( 1 + \frac{\ln r}{2\ln c} \right) \right\rceil - 1
\]
steps.

Let $\alpha > c$. Let us denote $\Gamma_k$ the volume ratio after $k$ steps and $\alpha_k = \frac{\Gamma_k^{1/r}}{\sqrt{r}}$. Then the $C$-norm of some row is at least $\alpha_k$, because otherwise the $2$-norms of all the rows are less than $\alpha_k \sqrt{r}$ and the volume ratio is less than $(\alpha_k \sqrt{r})^r = \Gamma_k$. By Lemma \ref{maxvol-lem}
\begin{align}
  \Gamma_{k+1} & \leqslant \Gamma_k / \alpha_k, \nonumber \\
  \alpha_{k+1}^r & \leqslant \alpha_k^r / \alpha_k, \nonumber \\
  \alpha_{k+1} & \leqslant \alpha_k / \alpha_k^{1/r} = \alpha_k^{1 - \frac{1}{r}}. \nonumber
\end{align}
If we start from $\alpha_0 = \alpha$, then after $k$ steps
\[
  \alpha_k \leqslant \alpha^{(1-\frac{1}{r})^k}.
\]
Assume $\alpha$ becomes less or equal to $c$ after $k_2$ iterations. Then we get the following condition:
\begin{equation}\label{k2-eq}
\begin{gathered}
  {\alpha ^{{{\left( {1 - \frac{1}{r}} \right)}^{{k_2}}}}} \leqslant c, \hfill \\
  {k_2} = \left\lceil {\frac{{\ln \frac{{\ln c}}{{\ln \alpha }}}}{{\ln \left( {1 - \frac{1}{r}} \right)}}} \right\rceil  = \left\lceil {\frac{{\ln \frac{{\ln \alpha }}{{\ln c}}}}{{\ln \left( {\frac{r}{{r - 1}}} \right)}}} \right\rceil  \leqslant \left\lceil {r\ln \frac{{\ln \frac{{{\Gamma^{2/r}}}}{r}}}{{\ln c}}} \right\rceil . \hfill \\ 
\end{gathered} 
\end{equation}

Adding $k_1$, in view of the rounding, we obtain
\begin{equation}\label{it_est}
k \leqslant r\left( {1 + \frac{{\ln r}}{{2\ln c}} + \max \left( {0,\ln \frac{{\ln \frac{{{\Gamma^{2/r}}}}{r}}}{{\ln c}}} \right)} \right).
\end{equation}

Let's consider how big the ratio $\Gamma$ can be for a randomly selected starting submatrix. Let us apply the {\myfont maxvol} algorithm to a matrix with randomly distributed matrices of left and right singular vectors.
Since the Gaussian matrices are unitary invariant and $\Gamma$ equals the volume ratio, the current matrix $C$ can be considered Gaussian, multiplied by some $r \times r$ matrix, which does not affect the volume ratios.

The probability of the volume of a random Gaussian matrix to differ from its expectation is exponentially small.
Even though the maximum is taken over $M^r N^r$ different matrices, we obtain with a high probability $\Gamma = O(r^r \ln^r MN)$.
If $c = const > 1$, the substitution in (\ref{it_est}) yields
\[
  k = O\left(r (\ln r + \ln \ln \ln MN)\right) \approx O(r \ln r)
\]
and $iter \approx O(r \ln r \cdot IT)$.

Of course, it is better to have a guarantee on $\Gamma$. We already have a guarantee $\Gamma \leqslant r!$ from the proposition \ref{pre-prop}.

Thus if we apply {\myfont pre-maxvol} beforehand,
\[
  k_2 \leqslant \left\lceil {r\ln \frac{{\ln r}}{{\ln c}}} \right\rceil.
\]
The asymptotics of the number of steps {\myfont maxvol} takes is then dominated by $k_1$, so
\[
  iter = O(r \ln r \cdot IT)
\]
is guaranteed when $c = const > 1$.

In practice, the starting submatrix for {\myfont maxvol} is instead often chosen using the Bebendorf cross algorithm \cite{Beb}, which is equivalent to incomplete Gaussian elimination with partial pivoting. Although the upper bound for elements in $C$ in this algorithm is $2^r$, it is usually enough. 
Indeed, we have $\Gamma \leqslant r^{r/2} \cdot 2^{r^2}$ and the expression for the additional number of steps $k_2$ (\ref{k2-eq}) provides that for $c = O(1)$ even such an estimate on the starting submatrix does not spoil the asymptotics for the number of permutations in {\myfont maxvol}.

\subsection{How to add more rows or columns}\label{maxvol2subsec}

Now consider the addition of rows after $r$-th one, which is performed by \cite{maxvol2} (see figure \ref{maxvol2-fig}).
As in {\myfont pre-maxvol}, the row lengths provide the addition criterion.

\begin{lemma}[\cite{maxvol2}]\label{maxvol2-lem}
Let $\hat A_0 \in \mathbb{C}^{n \times r}$ be a submatrix in the first $n$ rows of the matrix $A \in \mathC^{M \times r}$. Then adding $i$-th row of $A$ to the submatrix $\hat A_0$ changes the squared volume of $\hat A_0$ as
\[
  \cV(\hat A)^2 / \cV(\hat A_0)^2 = 1 + (l_0)_i = 1 + \|(C_0)_{i,:}\|_2^2,
\]
where
\[
  C_0 = A \hat A_0^{+} \in \mathC^{M \times n}.
\]
\end{lemma}

The original {\myfont maxvol2} algorithm \cite{maxvol2} is presented below including the derivation of the update formulas.

We start with the notation. Again, we deal with the matrix $A \in \mathC^{M \times r}$. Its submatrix $\hat A \in \mathbb{C}^{(n+1) \times r}$ expands the submatrix $\hat A_0 \in \mathbb{C}^{n \times r}$ by appending a row $a^* = A_{i,:}$.
\[
\hat A = \left[ {\begin{array}{*{20}{c}}
  {{{\hat A}_0}} \\ 
  {{a^ * }} 
\end{array}} \right].
\]
In order to add a row maximizing the volume, we should be able to update the matrix $C_0 \in \mathC^{M \times n}$ and the squared 2-norms of its rows. They are stored in a vector $l_0 \in \mathC^M$. For this we define a matrix $C \in \mathC^{M \times n}$ and a column $C' \in \mathC^M$ as follows:
\[
A{\hat A^ + } = \left[ {\begin{array}{*{20}{c}}
  C&{C'} 
\end{array}} \right],\quad \hat A = \left[ {\begin{array}{*{20}{c}}
  {{{\hat A}_0}} \\ 
  a^* 
\end{array}} \right].
\]
They can also be expressed as
\begin{equation}\label{ccs-eq}
C = A{\left( {{{\hat A}^ * }\hat A} \right)^{ - 1}}{{\hat A}_0^*},\quad C' = A{\left( {{{\hat A}^ * }\hat A} \right)^{ - 1}}a
\end{equation}

Our task is to calculate $C$ and $C'$ on the basis of $C_0$. The expression (\ref{ccs-eq}) shows that it is sufficient to find $\left( {{{\hat A}^ * }\hat A} \right)^{ - 1}$. Since
\[
{{\hat A}^ * }\hat A = \hat A_0^ * {{\hat A}_0} + aa^* = \left( {\hat A_0^ * {{\hat A}_0}} \right)\left( {I + {{\left( {\hat A_0^ * {{\hat A}_0}} \right)}^{ - 1}}aa^*} \right),
\]
then for the inverse
\begin{align}\label{asa-eq}
{\left( {{{\hat A}^ * }\hat A} \right)^{ - 1}} & = {\left( {\hat A_0^ * {{\hat A}_0}} \right)^{ - 1}}{\left( {I + {{\left( {\hat A_0^ * {{\hat A}_0}} \right)}^{ - 1}}aa^*} \right)^{ - 1}} \nonumber \\
  & = {\left( {\hat A_0^ * {{\hat A}_0}} \right)^{ - 1}}\left( {I - \frac{{{{\left( {\hat A_0^ * {{\hat A}_0}} \right)}^{ - 1}}aa^*}}{{1 + a^*{{\left( {\hat A_0^ * {{\hat A}_0}} \right)}^{ - 1}}a}}} \right).
\end{align}
The expression in the brackets can be simplified by introducing the notation $c^* = C_{i,:}$. 
Indeed,
\begin{align}
  c^* & = {C_{0i,:}} = {A_{i,:}}\hat A_0^ +  = a^*{\left( {\hat A_0^ * {{\hat A}_0}} \right)^{ - 1}}\hat A_0^ *, \nonumber \\
  c^*c & = a^*{\left( {\hat A_0^ * {{\hat A}_0}} \right)^{ - 1}}\hat A_0^ * \hat A_0 {\left( {\hat A_0^ * {{\hat A}_0}} \right)^{ - 1}} a = a^* {\left( {\hat A_0^ * {{\hat A}_0}} \right)^{ - 1}} a. \nonumber
\end{align}
Also,
\[
{\left( {\hat A_0^ * {{\hat A}_0}} \right)^{ - 1}}aa^* = aa^*{\left( {\hat A_0^ * {{\hat A}_0}} \right)^{ - 1}}.
\]
Substituting into (\ref{asa-eq}) gives
\[
{\left( {{{\hat A}^ * }\hat A} \right)^{ - 1}} = {\left( {\hat A_0^ * {{\hat A}_0}} \right)^{ - 1}}\left( {I - \frac{{aa^*{{\left( {\hat A_0^ * {{\hat A}_0}} \right)}^{ - 1}}}}{{1 + c^*c}}} \right).
\]
Taking into account
\[
C_0 c = A{\left( {\hat A_0^ * {{\hat A}_0}} \right)^{ - 1}}\hat A_0^ * {{\hat A}_0}{\left( {\hat A_0^ * {{\hat A}_0}} \right)^{ - 1}}a = A{\left( {\hat A_0^ * {{\hat A}_0}} \right)^{ - 1}}a,
\]
one can calculate $C'$ from (\ref{ccs-eq}):
\begin{align}
C' & = A{\left( {\hat A_0^ * {{\hat A}_0}} \right)^{ - 1}}\left( {I - \frac{{{a^ * }a{{\left( {\hat A_0^ * {{\hat A}_0}} \right)}^{ - 1}}}}{{1 + c^*c}}} \right)a \nonumber \\
  & = A{\left( {\hat A_0^ * {{\hat A}_0}} \right)^{ - 1}}\left( {a - \frac{a c^*c}{{1 + c^*c}}} \right)a \nonumber \\
  & = A{\left( {\hat A_0^ * {{\hat A}_0}} \right)^{ - 1}}\frac{{a}}{{1 + c^*c}} \nonumber \\
  & = \frac{{{C_0}c}}{{1 + c^*c}} \nonumber \\
  & = C_0 C_{0i,:}^* / (1 + l_i). \nonumber
\end{align}

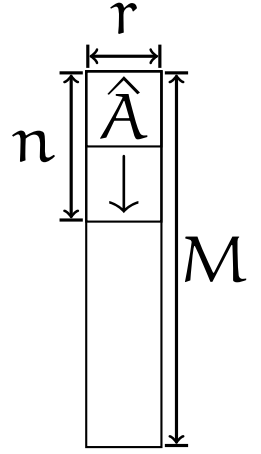
\begin{wrapfigure}{R}{0.25\textwidth}
    \begin{tikzpicture}
	\makeatletter
		\let\cs\@undefined
	\makeatother

	\newlength{\cs} 
	\setlength{\cs}{1.0cm}

	\foreach \x in {1.5\cs}
		\foreach \y in {1\cs}
	{\huge
		\draw [thick] (\x, \y) rectangle +(1\cs, 5\cs);
		\draw [thick] (\x, \y + 3\cs) rectangle +(1\cs, 2\cs);
		\draw [thick] (\x, \y + 4\cs) rectangle +(1\cs, 1\cs);

        \node at (\x + 0.5\cs, \y + 4.5\cs) {\( \hat A \)};
        \node at (\x + 0.5\cs, \y + 3.5\cs) {\( \downarrow \)};
		
		\node at (\x + 0.5\cs, \y + 5.7\cs) {\( r \)};
		\node at (\x - 0.7\cs, \y + 4\cs) {\( n \)};
		\node at (\x + 1.7\cs, \y + 2.5\cs) {\( M \)};
		
		\draw [|<->|, very thick] (\x, \y + 5.2\cs) -- (\x + 1\cs, \y + 5.2\cs);
		\draw [|<->|, very thick] (\x - 0.2\cs, \y + 3\cs) -- (\x - 0.2\cs, \y + 5\cs);
		
		\draw [|<->|, very thick] (\x + 1.2\cs, \y) -- (\x + 1.2\cs, \y + 5\cs);
	}
\end{tikzpicture}
    \caption{Extension of $\hat A = A_{\mathcal{I},:}$ by {\myfont maxvol2}.}
    \label{maxvol2-fig}
\end{wrapfigure}

We compute $C$ similarly:
\begin{align}
  C & = A{\left( {\hat A_0^ * {{\hat A}_0}} \right)^{ - 1}}\left( {I - \frac{{aa^*{{\left( {\hat A_0^ * {{\hat A}_0}} \right)}^{ - 1}}}}{{1 + c^*c}}} \right)\hat A_0^ *  \hfill \\
   & = {C_0} - C'a^*{\left( {\hat A_0^ * {{\hat A}_0}} \right)^{ - 1}}\hat A_0^ *  \hfill \\
   & = {C_0} - C'c^* \nonumber \\ 
   & = C_0 - C' C_{0i,:}. \nonumber
\end{align} 

Now we can directly calculate the lengths of the new rows. The updated $l_0$ is denoted by $l$.
\begin{align}
  {l_j} & = {C_{j,:}}C_{j,:}^ * + C'_j C'^*_j  \nonumber \\
   & = \left( {{C_{0j,:}} - {C'_j}{C_{0i,:}}} \right){\left( {{C_{0j,:}} - {C'_j}{C_{0i,:}}} \right)^ * } + {\left| {{C'_j}} \right|^2} \nonumber \\
   & = {l_{0j}} - 2{C'_j}^ * {C_{0j,:}}C_{0i,:}^ *  + {\left| {{C'_j}} \right|^2}{C_{0i,:}}C_{0i,:}^ *  + {\left| {{C'_j}} \right|^2} \nonumber \\
   & = {l_{0j}} - 2C'^*_j{C'_j}\left( {1 + {l_{0i}}} \right) + {\left| {{C'_j}} \right|^2}{l_{0i}} + {\left| {{C'_j}} \right|^2} \nonumber \\
   & = {l_{0j}} - {\left| {{C'_j}} \right|^2}\left( {1 + {l_{0i}}} \right) \nonumber
\end{align}

We present here full version of the original {\myfont maxvol2} algorithm from \cite{maxvol2}. We also derive a new version in the Appendix, which is based on the Householder reflections. It has a complexity $O\left(Mr(n-r)\right)$ instead of $O\left(M(n^2-r^2)\right)$ \cite{Mich}.

\begin{algorithm}[h!]
\caption{Original {\myfont maxvol2} \cite{maxvol2}}
\label{maxvol2-alg}
\begin{algorithmic}[1]
\REQUIRE{Matrix $A \in \mathbb{C}^{M \times r}$, the starting set of row indices $\mathcal{I}$ of cardinality $r$, the required final size $n$.}
\ENSURE{$\mathcal{I}$, supplemented by $n-r$ row indices chosen greedily to maximize the volume.}
\STATE $C := A [A_{\mathcal{I},:}^{-1}, 0_{r \times (n-r)}]$
\STATE $current\_order = \{1, \ldots, M\}$
\STATE $C.swap(\mathcal{I}, \{1, \ldots, r\}, current\_order)$
\COMMENT{$C.swap(A,B,order)$ swaps the elements of $C$, corresponding to indexes from $A$ and $B$ ($A_i$ swaps with $B_i$) and changes the order of corresponding indexes in $order$.
}
\STATE $l := 0_M$
\FOR{$i := r+1$ \TO $M$}
  \STATE $l_i := \| C_{i,:} \|_2^2$
\ENDFOR
\FOR{$new\_size := r+1$ \TO $n$}
  \STATE $i := \mathop {\arg \max }\limits_i  {{l_i}}$
  \STATE $l_i' := 1 + l_i$
  \STATE $l_i := 0$
  \STATE $C_I' := C_{i,:}^* / l_i'$
  \STATE $C' := C C_I'$
  \STATE $C := C - C' C_{i,:}$ \label{m2C}
  \FOR{$j := new\_size$ \TO $M$}
    \STATE $l_j := l_j - l_i' |C_{j}'|^2$
  \ENDFOR
  \STATE $C_{:,new\_size} := C'$
  \STATE $C.swap(new\_size,i,current\_order)$
  \STATE $l.swap(new\_size,i)$
\ENDFOR
\STATE $\mathcal{I} := current\_order[1..n]$
\end{algorithmic}
\end{algorithm}

This algorithm extends the submatrix by greedily chosen rows to maximize the volume. The matrix $C$ can be obtained from {\myfont maxvol} (algorithm \ref{maxvol-alg}) or {\myfont pre-maxvol} (algorithm \ref{premaxvol-alg}). The transposed version can be applied to select columns in some rows $R$.

Also, {\myfont maxvol2} provides some guarantees on the resulting submatrix volume.

\begin{proposition}\label{maxvol2-prop}
If the algorithm {\myfont maxvol2} is applied to $r \times r$ submatrix that differs in volume from the maximum among all $r \times r$ submatrices by no more than $\Gamma = c^r r^{r/2}$ times, where $c \geqslant 1$ (for example, after the algorithm {\myfont maxvol} with the parameter $c$), then after applying {\myfont maxvol2} the ratio of the maximum 2-volume $V_{max}$ among all $n \times r$ submatrices to the 2-volume of the found by {\myfont maxvol2} submatrix $\hat A \in \mathC^{n \times r}$ does not exceed
\begin{equation}\label{after_maxvol2}
  V_{max}/\cV(\hat A) \leqslant \max \left( \left( e\frac{n}{r} \right)^{r/2}, (c^2 r)^{r/2} \right).
\end{equation}
\end{proposition}

\begin{proof}
The volume ratio of each $r \times r$ submatrix to the maximum volume among $n \times r$ submatrices is not greater than $c^r r^{r/2}$ and there are $C_n^r \leqslant (n/r)^r$ of $r \times r$ different submatrices inside any $n \times r$ submatrix. By the Cauchy-Binet formula squared volume of $n \times r$ submatrix is equal to the sum of squared volumes of all $r \times r$ submatrices. So the volume ratio of the initial $r \times r$ submatrix to $n \times r$ submatrix of maximum volume is not greater than $c^r n^{r / 2}$.

Let $x \geqslant 1 / \sqrt{r}$. If the ratio of the current $k \times r$ ($r \leqslant k < n$) submatrix volume to the maximum volume among $n \times r$ submatrices is equal to $x^r n^{r / 2}$, at each step the ratio of the current volume to the maximum volume decreases by no less than
\[
  \sqrt{1 + x^2r} \geqslant x \sqrt{r}
\]
times, since it is possible to take rows from the maximum volume submatrix. Hence
\[
  x^r n^{r/2} \leqslant \frac{c^r n^{r/2}}{\left( x \sqrt{r} \right)^{n-r}},
\]
\[
  c^r \geqslant x^n r^{\frac{n-r}{2}}.
\]
If the ratio of volumes turns out to be larger than $c^r r^{r/2}$, then
\[
  x^r n^{r/2} > c^r r^{r/2} \geqslant x^n r^{n/2},
\]
\[
  x^{n-r} < n^{r/2} / r^{n/2},
\]
\[
  x n^{1/2} < \frac{n}{r} \cdot \left( \frac{n}{r} \right)^{\frac{r}{n-r}} \leqslant e \frac{n}{r},
\]
which proves (\ref{after_maxvol2}).
\end{proof}

\subsection{Rectangular locally maximum volume search}\label{maxvolrectsubsec}

Now we are ready to move on to the new algorithms, and we start with {\myfont Dominant-C} (see figure \ref{dominantc-fig}). The main idea is to find a dominant rectangular submatrix. Technically it is done by generalizing the updates in the original {\myfont maxvol2} algorithm \cite{maxvol2}.
Because the {\myfont Dominant-C} searches for a locally maximum volume, it has similar properties to {\myfont maxvol}. For instance, the 2-norm of any row of $C \in \mathC^{M \times n}$ outside the selected ones does not exceed $\sqrt{\frac{r}{n - r + 1}}$. 
It also makes the item 2 of Lemma \ref{first-lem} and the Corollary \ref{maxvol2-cons} true for the submatrix, found by {\myfont Dominant-C}.

The following Lemma provides the replacement criterion. 

\begin{lemma}\label{bij-lem}
Let $\hat A_0 \in \mathbb{C}^{n \times r}$ be a submatrix in the first $n$ rows of the matrix $A \in \mathC^{M \times r}$. 
Then, replacing $j$-th row of $\hat A_0$ by the $i$-th row of $A$ (for $i > n$) changes the squared volume of $\hat A_0$ by a factor
\[
  B_{ij} = |C_{0ij}|^2 + (1 + l_{0i})(1 - l_{0j}),
\]
where the matrix $C_0$ and the vector $l_0$ denote the same as in the Lemma \ref{maxvol2-lem}.
\end{lemma}

\begin{proof}
To prove the Lemma, we need to derive the fast update formulas. 
For this, we use the update formulas from the original {\myfont maxvol2} (algorithm \ref{maxvol2-alg}).

The update is done in 2 steps. 
First, we append the $i$-th row, and then remove $j$-th row. 
Since we already know how to append, let us only focus on the removal.

Firstly, note that appending the $i$-th row reduces its elements $1 + l_{0i}$ times. 
To check this, substitute $i$ in the update formula for $C$ (line \ref{m2C} in the original {\myfont maxvol2} algorithm). 
Removing the $j$-th row increases the elements $1 + \tilde l_j$ times (the tilde denotes the value after the removal).
\begin{equation}\label{cjt-eq}
  \tilde C_{j,:} = C_{j,:} (1 + \tilde l_j)
\end{equation}

From the update formula for $l$, we get that
\begin{align}
  l_j & = \tilde l_j - |C_j'|^2 (1 + \tilde l_j) \nonumber \\
       & = \tilde l_j - |C_{j,:} C_{j,:}^*|^2 / (1 + \tilde l_j) \nonumber \\
       & = \tilde l_j - \tilde l_j^2 / (1 + \tilde l_j) \nonumber \\
       & = \frac{\tilde l_j}{1 + \tilde l_j}, \nonumber
\end{align}
which means
\begin{equation}\label{eq:l-trans}
1 + \tilde l_j = 1 / (1 - l_j).
\end{equation}

Using (\ref{cjt-eq}) and (\ref{eq:l-trans}) we obtain a formula for the update of $C$:
\begin{eqnarray}
  C & := & \tilde C - C' \tilde C_{j,:} = \tilde C - C' C_{j,:} (1 + \tilde l_j) \nonumber \\
  \tilde C & := & C + C' C_{j,:}/(1-l_j) \nonumber
\end{eqnarray}

Similarly, the update of $l$ is calculated as
\begin{eqnarray}
  l_k & := & \tilde l_k - |C'|^2 (1 + \tilde l_j) \nonumber \\
  \tilde l_k & := & l_k + |C'|^2 / (1 - l_j) \nonumber
\end{eqnarray}

After adding $i$-th row, $l_j$ depends on $l_{0j}$ as:
\begin{equation}\label{eq:lj-first}
  l_j = l_{0j} - |C'|^2 (1 + l_{0i}) = l_{0j} - |C_{0i,:} C_{0j,:}^*|^2 / (1 + l_{0i}).
\end{equation}

The product $C_{i,:} C_{j,:}^*$ is written in the element $C_{ij}$. Indeed,
\[
{C_{0i,:}}C_{0j,:}^ *  = {A_{i,:}}{\hat A_0^ + }{\left( {{{\hat A_0}^ + }} \right)^ * }A_{j,:}^ *  = {A_{i,:}}{\left( {{{\hat A_0}^ * }\hat A_0} \right)^{ - 1}}A_{j,:}^ *  = {A_{i,:}} \hat A_{0j,:}^ +  = {C_{0ij}}.
\]
Therefore, the addition of the $i$-th row changes the $j$-th row length as
\begin{equation}\label{eq:lj-last}
  l_j = l_{0j} - |C_{0ij}|^2 / (1 + l_{0i}).
\end{equation}

With the help of the update formulas, we can calculate the ratio of the volumes after the interchange. 
From the Lemma \ref{maxvol2-lem}, we know that adding the $i$-th row increases the volume $1 + l_{0i}$ times. The expression (\ref{eq:l-trans}) shows that removing the $j$ -th row results in the volume decrease $1/(1-l_j)$ times. The total change considering (\ref{eq:lj-last}) is equal to
\[
  B_{ij} := (1 + l_{0i})(1 - l_j) = |C_{0ij}|^2 + (1 + l_{0i})(1 - l_{0j}).
\]
\end{proof}

The use of the matrix $B_{ij}$ for the decision to replace the columns leads us to the next algorithm.

The complexity of the algorithm is $O(Mn \cdot iter)$.

\begin{algorithm}[h!]
\caption{{\myfont Dominant-C}}
\label{dominantc-alg}
\begin{algorithmic}[1]
\REQUIRE{Matrix $A \in \mathbb{C}^{M \times r}$, the starting set of row indices $\mathcal{I}$ of cardinality $n$. For example, $\mathcal{I} = \{1, ..., n \}$.}
\ENSURE{The updated set $\mathcal{I}$, corresponding to a dominant submatrix.}
\STATE $C := A A_{\mathcal{I},:}^+$
\FOR{$i := 1$ \TO $M$}
  \STATE $l_i := \| C_{i,:} \|_2^2 + 1$
\ENDFOR
\FOR{$i := n+1$ \TO $M$}
  \FOR{$j := 1$ \TO $n$}
    \STATE $B_{i,j} := |C_{i,j}|^2 + l_i (1 - C_{j,j})$
  \ENDFOR
\ENDFOR
\STATE $\{i,j\} := \mathop {\arg \max }\limits_{i,j} {B_{i,j}}$
\WHILE{$B_{i,j} > 1$}
  \STATE Update $C$, $l$ and $B$
  \STATE Replace $j$ with $i$ in $\mathcal{I}$
  $\{i,j\} := \mathop {\arg \max }\limits_{i,j} {B_{i,j}}$
\ENDWHILE
\end{algorithmic}
\end{algorithm}

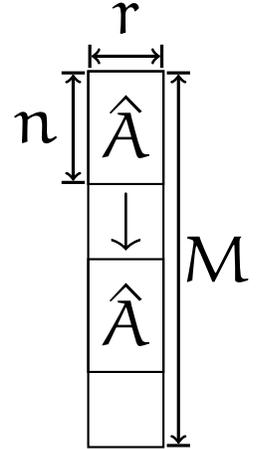
\begin{wrapfigure}{R}{0.25\textwidth}
    \begin{tikzpicture}
	\makeatletter
		\let\cs\@undefined
	\makeatother

	\newlength{\cs} 
	\setlength{\cs}{1.0cm}

	\foreach \x in {1.5\cs}
		\foreach \y in {1\cs}
	{\huge
		\draw [thick] (\x, \y) rectangle +(1\cs, 5\cs);
		\draw [thick] (\x, \y + 3.5\cs) rectangle +(1\cs, 1.5\cs);
		\draw [thick] (\x, \y + 1\cs) rectangle +(1\cs, 1.5\cs);

        \node at (\x + 0.5\cs, \y + 4.25\cs) {\( \hat A \)};
        \node at (\x + 0.5\cs, \y + 1.75\cs) {\( \hat A \)};
        \node at (\x + 0.5\cs, \y + 3\cs) {\( \downarrow \)};
		
		\node at (\x + 0.5\cs, \y + 5.7\cs) {\( r \)};
		\node at (\x - 0.7\cs, \y + 4.25\cs) {\( n \)};
		\node at (\x + 1.7\cs, \y + 2.5\cs) {\( M \)};
		
		\draw [|<->|, very thick] (\x, \y + 5.2\cs) -- (\x + 1\cs, \y + 5.2\cs);
		\draw [|<->|, very thick] (\x - 0.2\cs, \y + 3.5\cs) -- (\x - 0.2\cs, \y + 5\cs);
		
		\draw [|<->|, very thick] (\x + 1.2\cs, \y) -- (\x + 1.2\cs, \y + 5\cs);
	}
\end{tikzpicture}
    \caption{Update of $\hat A = A_{\mathcal{I},:}$ through the algorithm {\myfont Dominant-C}.}
    \label{dominantc-fig}
\end{wrapfigure}

Let's estimate the number of iterations after {\myfont pre-maxvol}.

After appending $n-r$ arbitrary rows, we get a starting $\hat A \in \mathC^{n \times r}$ submatrix for {\myfont Dominant-C} with the volume which differs from the maximum not more than $\Gamma = (rn)^{r/2}$ times (by the same reasoning as at the beginning of the proposition \ref{maxvol2-prop} proof). 
The use of {\myfont maxvol} or {\myfont maxvol2} does not improve the asymptotics, since we take logarithm from $\Gamma$, like in {\myfont maxvol} algorithm (\ref{k2-eq}).

Let the maximum row length in $C$ be $x$. The same is true for $C' = C Q_{\hat A} = A R_{\hat A}^{-1}$, where $Q_{\hat A} \in \mathC^{n \times r}$ and $R_{\hat A} \in \mathC^{r \times r}$ are the $QR$ decomposition of $\hat A$. 
We further deal with $C'$, since it has only $r$ columns.

The volume of the extended submatrix is $1 + x^2$ times greater. 
The sum of 2-volumes of all $n \times r$ submatrices is $(n-r+1)(1+x^2)$. 
Among these submatrices, there exists one with a volume not less than
\[
  \frac{(n-r+1)(1+x^2)-1}{n} \geqslant x^2 \frac{n-r+1}{n}.
\]

Thus, if the initial volume ratio of the current submatrix to the maximal volume submatrix without common rows is equal to $V = \left (\alpha^2 \frac{n^2}{r (n - r + 1)} \right)^{r / 2}$, then with each step this ratio decreases in $\alpha$ times. 
So, until $\alpha = c$ (the calculation is the same as for {\myfont maxvol})
\[
 k_2 \leqslant \left\lceil {r\ln \frac{{\ln \frac{{{r(n-r+1) \Gamma^{2/r}}}}{n^2}}}{{\ln c}}} \right\rceil \leqslant \left\lceil {r\ln \frac{{2 \ln r}}{{\ln c}}} \right\rceil
\]
steps are required.

After that, the volume is reduced by $c$ times at each step. So, no more than
\[
{k_1} \leqslant \left\lceil {r\frac{{\ln {c^2}\frac{{{n^2}}}{{r\left( {n - r + 1} \right)}}}}{{2\ln c }}} \right\rceil  \leqslant \left\lceil {r\left( {1 + \frac{{\ln n}}{{2 \ln c}}} \right)} \right\rceil  - 1
\]
additional steps are required. Summarizing and taking into account that one rounding is enough, we obtain
\[
k \leqslant r\left( {1 + \frac{{\ln n}}{{2 \ln c}}} \right) + r\ln \left( \frac{2 \ln r}{\ln c} \right).
\]
If $c = const > 1$, then
\[
  it = O(r \ln n).
\]

In addition to being an intermediate step in the cross approximation, {\myfont Dominant-C} is interesting in itself. 
Firstly, it allows to guarantee the cross approximation accuracy estimates in the 2-norm  (if used in the first $r$ left and right singular vectors, see Theorem 2 of \cite{me}) and in the $C$-norm (Theorem 4.7 of \cite{Mich}), because these estimates are based on Lemma \ref{first-lem} and require locally maximum volume submatrices.
Such estimates are not guaranteed by the {\myfont maxvol2} algorithm even if some approximation of rank $r$ is known. 
Moreover, the limited step number makes the algorithm competitive with many algorithms from \cite{SubSel} for finding a submatrix with a small norm of the pseudoinverse. 
The comparison is presented in subsection \ref{appsubsec}. 

\begin{wrapfigure}{L}{0.45\textwidth}
    \begin{tikzpicture}
	\makeatletter
		\let\cs\@undefined
	\makeatother

	\newlength{\cs} 
	\setlength{\cs}{1.0cm}

	\foreach \x in {1.5\cs}
		\foreach \y in {1\cs}
	{\huge
		\draw [thick] (\x, \y) rectangle +(5\cs, 2\cs);
		\draw [thick] (\x, \y) rectangle +(1\cs, 2\cs);
		\draw [thick] (\x + 2\cs, \y) rectangle +(1\cs, 2\cs);

        \node at (\x + 0.5\cs, \y + 1\cs) {\( \hat A \)};
        \node at (\x + 1.5\cs, \y + 1\cs) {\( \rightarrow \)};
        \node at (\x + 2.5\cs, \y + 1\cs) {\( \hat A \)};
		
		\node at (\x - 0.7\cs, \y + 1\cs) {\( m \)};
		\node at (\x + 0.5\cs, \y + 2.7\cs) {\( r \)};
		\node at (\x + 2.5\cs, \y - 0.7\cs) {\( N \)};
		
		\draw [|<->|, very thick] (\x - 0.2\cs, \y) -- (\x - 0.2\cs, \y + 2\cs);
		\draw [|<->|, very thick] (\x, \y + 2.2\cs) -- (\x + 1\cs, \y + 2.2\cs);
		
		\draw [|<->|, very thick] (\x, \y - 0.2\cs) -- (\x + 5\cs, \y - 0.2\cs);
	}
\end{tikzpicture}
    \caption{Update of $\hat A = A_{:,\mathcal{J}}$ through the algorithm {\myfont Dominant-R}.}
    \label{dominantr-fig}
\end{wrapfigure}
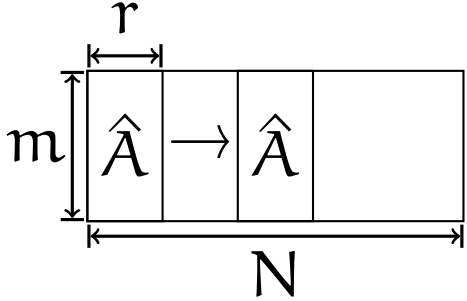

Now, to find the dominant rectangular submatrix, it only remains to present the algorithm {\myfont Dominant-R} (see figure \ref{dominantr-fig}). 
We need to maximize the volume by replacing one of the $r$ columns. 
It can be done with the $RRQR$ algorithm from \cite{rrqr}. We call this $RRQR$ algorithm $GEQR$ by the first letters of the authors' surnames.
{\myfont Dominant-R} is its modification with the same asymptotics, but with the reduced coefficient due to the specifics of the matrix size ($N \gg n$). We use the same notation as in \cite{rrqr} for the Lemma and the algorithm.

\begin{lemma}[\cite{rrqr}, Lemma 3.1]
Let
\[
M = Q\left[ {\begin{array}{*{20}{c}}
  A&B \\ 
  {}&C^T 
\end{array}} \right] \in {\mathbb{R}^{m \times N}},
\]
where $QA$ is a $QR$ decomposition of the submatrix $\hat A \in \mathbb{R}^{m \times r}$ of the matrix $M$.
Then, replacing the $j$-th column of $\hat A$ by the $i$-th column of $M$ changes the squared volume of $\hat A$ by a factor of
\[
  m_{ij} = |(A^{-1}B)_{ij}|^2 + \gamma_i \omega_j.
\]
Here $\gamma$ denotes the vector of the squared $2$-norms of the columns of $C^T$, and $\omega$ is the vector of the squared $2$-norms of the rows of $A^{-l}$.
\end{lemma}

We set the approximation rank, and perform column replacements. 
The new columns can be appended by {\myfont pre-maxvol}.

\begin{algorithm}[h!]
\caption{{\myfont Dominant-R} (The full version is a modification of the column replacement in $RRQR$ algorithm from \cite{rrqr})}
\label{dominantr-alg}
\begin{algorithmic}[1]
\REQUIRE{Matrix $M \in \mathbb{R}^{m \times N}$, the starting set of column indices $\mathcal{I}$ of cardinality $r$ (for example, $\mathcal{I} = \{1, ..., r \}$}), threshold parameter $f \geqslant 1$. 
\ENSURE{The updated set $\mathcal{I}$, corresponding to a dominant rectangular submatrix.}
\STATE $Q, A := QR(M_{:,\mathcal{I}})$
\STATE $Q := Q_{:,1:r}$
\FOR{$i := 1$ \TO $N-r$}
  \STATE $\gamma_i := \| M_{:,i+r} \|_2^2 - \| (Q^T M)_{:,i+r} \|_2^2$
\ENDFOR
\FOR{$i := 1$ \TO $r$}
  \STATE $\omega_i := \| A_{i,:}^{-1} \|_2^2$
\ENDFOR
\STATE $A^{-1}B := A^{-1} \cdot B$
\STATE $m := \omega^T \gamma$
\FOR{$i := 1$ \TO $r$}
  \FOR{$j := 1$ \TO $N-r$}
    \STATE $m_{i,j} := m_{i,j} + | (A^{-1}B)_{i+r,j+r}|^2$
  \ENDFOR
\ENDFOR
\STATE $\{i1,j1\} := \mathop {\arg \max }\limits_{i1,j1} {m_{i1,j1}}$
\WHILE{$m_{i1,j1} > f$}
  \STATE Replace $i1$ with $j1$ in $\mathcal{I}$
  \STATE Update $Q$ and $A$
  \STATE Update $\gamma$, $\omega$ and $A^{-1}B$
  \STATE $m := \omega^T \gamma$
  \FOR{$i := 1$ \TO $r$}
    \FOR{$j := 1$ \TO $N-r$}
      \STATE $m_{i,j} := m_{i,j} + | (A^{-1}B)_{i+r,j+r}|^2$
    \ENDFOR
  \ENDFOR
  \STATE $\{i1,j1\} := \mathop {\arg \max }\limits_{i1,j1} {m_{i1,j1}}$
\ENDWHILE
\end{algorithmic}
\end{algorithm}

Similar to {\myfont maxvol} we get that {\myfont Dominant-R} has complexity (see the notation in the algorithm) $O(Nmr^2 \frac {\ln r} {\ln f})$ if applied after {\myfont pre-maxvol}. This improves the original $GEQR$ \cite{rrqr} complexity of $O(Nmr^2 \frac {\ln N} {\ln f})$.

\begin{theorem}
Let (transposed) {\myfont pre-maxvol} and {\myfont Dominant-R} be successively applied to the matrix $A \in \mathbb{R}^{m \times N}$. Then after $O(r \log_f r)$ steps {\myfont Dominant-R} outputs a submatrix $\hat A \in \mathbb{R}^{m \times r}$ with $QR$ decomposition $\hat A = \tilde Q \tilde R$ such that $Q = \tilde Q$ and $R = \tilde Q^T A$ are a strong rank-revealing $QR$ decomposition with a factor $f$ (see definition in \cite{rrqr}) up to column permutation.
\end{theorem}
\begin{proof}
If {\myfont Dominant-R} with the parameter $f$ stops, then the output submatrix $\hat A$ indeed produces a strong rank-revealing $QR$ decomposition, because the update criteria is the same as in $GEQR$ \cite{rrqr}.

The number of steps is limited, because {\myfont pre-maxvol} outputs a submatrix with the volume ratio $\Gamma \leqslant r!$ to the maximum volume among all $m \times r$ (we apply {\myfont pre-maxvol} to $A^T$, so $m \geqslant r$) submatrices and the squared volume ratio is decreased at least by a factor $f$ at each step, so the number of iterations is bounded by
\[
  iter \leqslant \log_f \Gamma^2 \leqslant \log_f \left( r! \right)^2 = O(r \log_f r).
\]
\end{proof}

The total complexity follows from $O(Nmr)$ complexity of each {\myfont Dominant-R} iteration (see full version description in the Appendix). And the complexity of {\myfont pre-maxvol} is lower.

Now we combine the algorithms {\myfont Dominant-C} and {\myfont Dominant-R} into a form similar to the one for {\myfont maxvol}. 
This combination allows us to search for the dominant rectangular submatrix in the entire matrix. 

The algorithm {\myfont maxvol-rect} consists of two parts. 
We search for a large volume submatrix in the fixed rows and a large volume submatrix in the fixed columns (see figure \ref{maxvolrect-fig}). 
Without loss of generality, we assume that there are more rows than columns.
Otherwise, the algorithm differs only in transposition.

\begin{figure}[H]
\center
    \begin{tikzpicture}
	\makeatletter
		\let\cs\@undefined
	\makeatother

	\newlength{\cs} 
	\setlength{\cs}{1.0cm}

	\foreach \x in {1.5\cs}
		\foreach \y in {1\cs}
	{\huge
		\draw [thick] (\x, \y) rectangle +(5\cs, 5\cs);
		\draw [thick] (\x + 1\cs, \y) rectangle +(1\cs, 5\cs);
		\draw [thick] (\x + 3\cs, \y) rectangle +(1\cs, 5\cs);
		\draw [thick] (\x, \y + 3\cs) rectangle +(5\cs, 1.5\cs);
		\draw [thick] (\x, \y + 0.5\cs) rectangle +(5\cs, 1.5\cs);

        \node at (\x + 3.5\cs, \y + 3.75\cs) {\( \hat A \)};
		\node at (\x + 1.5\cs, \y + 1.25\cs) {\( \hat A \)};
		\node at (\x + 3.5\cs, \y + 1.25\cs) {\( \hat A \)};
		\node at (\x + 3.5\cs, \y + 2.5\cs) {\( \downarrow \)};
		\node at (\x + 2.5\cs, \y + 1.25\cs) {\( \leftarrow \)};
		
		\node at (\x + 1.5\cs, \y + 5.7\cs) {\( r \)};
		\node at (\x + 5.7\cs, \y + 1.25\cs) {\( n \)};
		\node at (\x - 0.7\cs, \y + 2.5\cs) {\( M \)};
		\node at (\x + 2.5\cs, \y - 0.7\cs) {\( N \)};
		
		\draw [|<->|, very thick] (\x + 1\cs, \y + 5.2\cs) -- (\x + 2\cs, \y + 5.2\cs);
		\draw [|<->|, very thick] (\x + 5.2\cs, \y + 0.5\cs) -- (\x + 5.2\cs, \y + 2\cs);
		
		\draw [|<->|, very thick] (\x - 0.2\cs, \y) -- (\x - 0.2\cs, \y + 5\cs);
		\draw [|<->|, very thick] (\x, \y - 0.2\cs) -- (\x + 5\cs, \y - 0.2\cs);
	}
\end{tikzpicture}
    \caption{Update of  current $\hat A = A_{\mathcal{I},\mathcal{J}}$ through the algorithm {\myfont maxvol-rect}.}
    \label{maxvolrect-fig}
\end{figure}
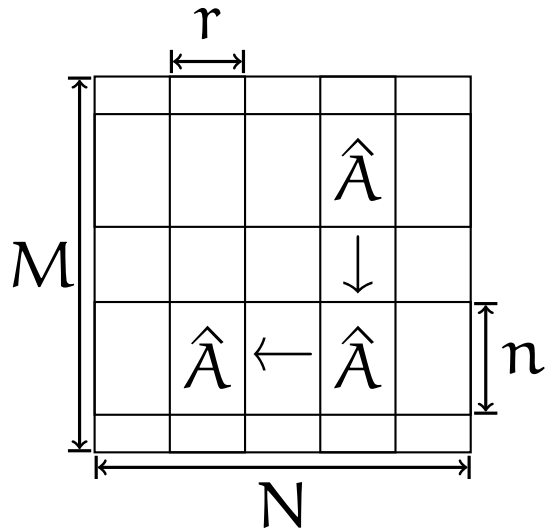

\begin{algorithm}[H]
\caption{{\myfont maxvol-rect}}
\label{maxvolrect-alg}
\begin{algorithmic}[1]
\REQUIRE{Matrix $A \in \mathbb{R}^{M \times N}$, the starting sets of row indices $\mathcal{I}$ and column indices $\mathcal{J}$ of cardinality $n$ and $r$ respectively.}
\ENSURE{The updated sets $\mathcal{I}$ and $\mathcal{J}$, corresponding to the dominant rectangular submatrix of rank $r$.
}
\STATE \COMMENT{In the beginning {\myfont pre-maxvol} can be applied.}
\STATE $changed := $ \TRUE
\STATE $old\_changed := $ \TRUE
\WHILE{$changed$}
  \STATE $changed := $ \FALSE
  \FOR{$changes\_in$ \textbf{in} $\{\mathcal{I}, \mathcal{J}\}$}
    \IF{$changes\_in = \mathcal{I}$}
      \STATE $C := A_{:,\mathcal{J}}$
      \STATE $changed := \operatorname{Dominant-C}(C, \mathcal{I})$ \COMMENT{Here a small addition to the algorithm is needed, which returns whether there were exchanges}
      \STATE $old\_changed := changed$
    \ELSE
      \STATE $R := A_{\mathcal{I},:}$
      \STATE $changed := \operatorname{Dominant-R}(R, \mathcal{J})$
    \ENDIF
    \IF{((\NOT $changed$) \AND ($changes\_in = \mathcal{J}$))}
      \STATE $changed = $ \FALSE
      \STATE \textbf{break}
    \ENDIF
    \IF{((\NOT $old\_changed$) \AND ($changes\_in = \mathcal{I}$))}
      \STATE $changed = $ \FALSE
      \STATE \textbf{break}
    \ENDIF
  \ENDFOR
\ENDWHILE
\end{algorithmic}
\end{algorithm}

\subsection{Large projective volume search}\label{maxvolprojsubsec}

Finally, we search for $m \times n$ submatrix of large $r$-projective volume by using $r \times n$ and $m \times r$ submatrices of large volume. This is the {\myfont maxvol-proj} algorithm. The construction idea is illustrated in figure \ref{proj-pic}.

\begin{algorithm}[H]
\caption{{\myfont maxvol-proj}}
\label{maxvolproj-alg}
\begin{algorithmic}[1]
\REQUIRE{Matrix $A \in \mathbb{R}^{M \times N}$, starting sets of row indices $\mathcal{I}$ and column indices $\mathcal{J}$ columns $r$ and final sizes of $C$ and $R$ ($n$ and $m$, respectively).
}
\ENSURE{The set $\mathcal{I}$ of the row indices of $R$ and the set $\mathcal{J}$ of the column indices of $C$ that contain a submatrix of a large projective volume.
}
\STATE $\operatorname{maxvol-rect} (A, \mathcal{I}, -)$
\STATE $\operatorname{maxvol-rect} (A^T, \mathcal{J}, -)$
\end{algorithmic}
\end{algorithm}

The informal justification is as follows. Suppose that $\rank A = r$ and some $r \times n$ and $m \times r$ submatrices are submatrices of maximal volume. Then it is easy to see that a submatrix $\hat A$ at the intersection of the $m$ rows and $n$ columns is indeed a submatrix of the maximal $r$-projective volume.

Thus, we reduced the problem of finding a submatrix of a large projective volume to the search for two smaller submatrices of large 2-volume. Thankfully, we already know how to find rectangular submatrices of large 2-volume by using {\myfont maxvol-rect}.

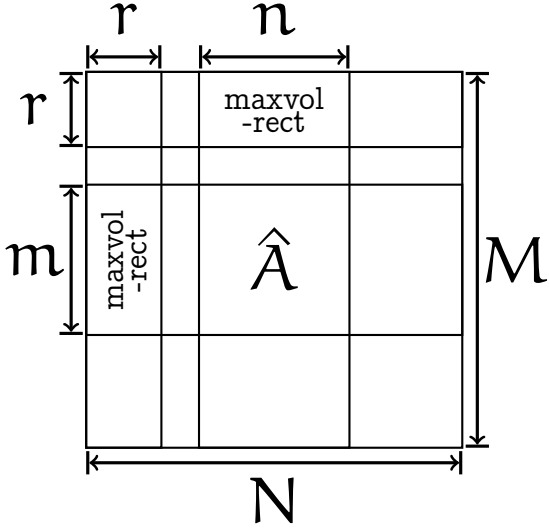
\begin{SCfigure}
    \begin{tikzpicture}
	\makeatletter
		\let\cs\@undefined
	\makeatother

	\newlength{\cs} 
	\setlength{\cs}{1.0cm}

	\foreach \x in {1.5\cs}
		\foreach \y in {1\cs}
	{\huge
		\draw [thick] (\x, \y) rectangle +(5\cs, 5\cs);
		\draw [thick] (\x, \y) rectangle +(1\cs, 5\cs);
		\draw [thick] (\x + 1.5\cs, \y) rectangle +(2\cs, 5\cs);
		\draw [thick] (\x, \y + 4\cs) rectangle +(5\cs, 1\cs);
		\draw [thick] (\x, \y + 1.5\cs) rectangle +(5\cs, 2\cs);

        \node at (\x + 2.5\cs, \y + 2.5\cs) {\( \hat A \)};
		
		\node at (\x + 0.5\cs, \y + 5.7\cs) {\( r \)};
		\node at (\x - 0.7\cs, \y + 2.5\cs) {\( m \)};
		\node at (\x - 0.7\cs, \y + 4.5\cs) {\( r \)};
		\node at (\x + 2.5\cs, \y + 5.7\cs) {\( n \)};
		\node at (\x + 5.7\cs, \y + 2.5\cs) {\( M \)};
		\node at (\x + 2.5\cs, \y - 0.7\cs) {\( N \)};
		
		{\tiny
		\node [align = center] at (\x + 2.5\cs, \y + 4.5\cs) {\normalsize maxvol\\ \normalsize -rect};
		\node [rotate = 90, align = center] at (\x + 0.5\cs, \y + 2.5\cs) {\normalsize maxvol\\ \normalsize -rect};
		}
		
		\draw [|<->|, very thick] (\x, \y + 5.2\cs) -- (\x + 1\cs, \y + 5.2\cs);
		\draw [|<->|, very thick] (\x + 1.5\cs, \y + 5.2\cs) -- (\x + 3.5\cs, \y + 5.2\cs);
		\draw [|<->|, very thick] (\x - 0.2\cs, \y + 4\cs) -- (\x - 0.2\cs, \y + 5\cs);
		\draw [|<->|, very thick] (\x - 0.2\cs, \y + 1.5\cs) -- (\x - 0.2\cs, \y + 3.5\cs);
		
		\draw [|<->|, very thick] (\x + 5.2\cs, \y) -- (\x + 5.2\cs, \y + 5\cs);
		\draw [|<->|, very thick] (\x, \y - 0.2\cs) -- (\x + 5\cs, \y - 0.2\cs);
	}
\end{tikzpicture}
    \caption{Submatrix $\hat A = A_{\mathcal{I},\mathcal{J}}$ returned by {\myfont maxvol-proj} algorithm.}\label{proj-pic}
\end{SCfigure}

\subsection{Improvements and simplifications}\label{impsubsec}

Algorithms {\myfont Dominant-C} and {\myfont Dominant-R} can work several times slower than {\myfont maxvol}. 
Firstly, $n$ appears instead of $r$ in the asymptotics. 
Secondly, the constant factor increases. 
Below we present an algorithm to find the large projective volume submatrix without {\myfont Dominant-R} (see figure \ref{cgr-fig}). 
It can be used either after {\myfont maxvol-proj} or independently.

\paragraph*{Algorithm {\myfont maxvol-proj} without {\myfont Dominant-R}}
\begin{algorithmic}[1]
\REQUIRE{Matrix $A \in \mathbb{C}^{M \times N}$, the starting sets of row indices $\mathcal{I}$ and column indices $\mathcal{J}$ of cardinality $r$, and final sizes $n$ and $m$ of $C$ and $R$ respectively.}
\ENSURE{The updated sets $\mathcal{I}$ (of the row indices of $R$) and $\mathcal{J}$ (of the column indices of $C$) corresponding to a submatrix of a large projective volume.
}
\STATE $changed := $ \TRUE
\STATE $old\_changed := $ \TRUE
\WHILE{$changed$}
  \STATE $changed := $ \FALSE
  \FOR{$changes\_in$ \textbf{in} $\{\mathcal{I}, \mathcal{J}\}$}
    \STATE $U, S, V := \operatorname{SVD} (A_{\mathcal{I},\mathcal{J}})$
    \IF{$changes\_in = \mathcal{I}$}
      \STATE $C := A_{:,\mathcal{J}} (V_{\{1, ..., r \},:})^T$
      \STATE $changed := \operatorname{Dominant-C}(C, \mathcal{I})$ \COMMENT{Here a small addition to the algorithm is needed, which returns whether there were exchanges}
      \STATE $old\_changed := changed$
    \ELSE
      \STATE $R := (U_{:,\{1, ..., r \}})^T A_{\mathcal{I},:}$
      \STATE $changed := \operatorname{Dominant-C}(R^T, \mathcal{J})$
    \ENDIF
    \IF{((\NOT $changed$) \AND ($changes\_in = \mathcal{J}$))}
      \STATE $changed = $ \FALSE
      \STATE \textbf{break}
    \ENDIF
    \IF{((\NOT $old\_changed$) \AND ($changes\_in = \mathcal{I}$))}
      \STATE $changed = $ \FALSE
      \STATE \textbf{break}
    \ENDIF
  \ENDFOR
\ENDWHILE
\end{algorithmic}

Multiplication of the columns by the first $r$ right singular vectors of $\hat A$ leads to the problem of finding the large $r$-projective volume submatrix in $C \hat A_r^+$. 
At the same time, the projective volume increases. 
Indeed, if
\[
  \cVr(\hat A' \hat A_r^+) > \cVr(\hat A \hat A_r^+) = 1,
\]
then since
\[
  \cVr(\hat A') \cVr(\hat A_r^+) \geqslant \cVr(\hat A \hat A_r^+),
\]
we get 
\[
  \cVr(\hat A') > 1 / \cVr(\hat A_r^+) = \cVr(\hat A).
\]

For the rows, the situation is similar. 
If $m = r$ or $n = r$, this algorithm is a simplified version of {\myfont maxvol-rect}. 
In this case the row replacements occur only if $|(A^{-1}B)_{ij}| > 1$, but not under the condition $m_{ij} > 1$, which holds more frequently.

Another version of the projective volume maximization \cite{me} is based on the application of the algorithm {\myfont maxvol2}. 

\begin{algorithm}[H]
\caption{Fast $CGR$ approximation search (2 maxvol2)}\label{fastcgr-alg}
\begin{algorithmic}[1]
\REQUIRE{Matrix $A \in \mathbb{C}^{M \times N}$, the starting sets of row indices $\mathcal{I}$ and column indices $\mathcal{J}$ of cardinality $r$, and final sizes $n$ and $m$ of $C$ and $R$ respectively.}
\ENSURE{The updated sets $\mathcal{I}$ (of the row indices of $R$) and $\mathcal{J}$ (of the column indices of $C$) corresponding to a submatrix of a large projective volume.
}
\STATE $\operatorname{maxvol} (A, \mathcal{I}, \mathcal{J})$
\STATE $C := A_{:,\mathcal{J}}$
\STATE $R := A_{\mathcal{I},:}$
\STATE $\operatorname{maxvol2} (C, \mathcal{I}, n)$
\STATE $\operatorname{maxvol2} (R^T, \mathcal{J}, m)$
\end{algorithmic}
\end{algorithm}

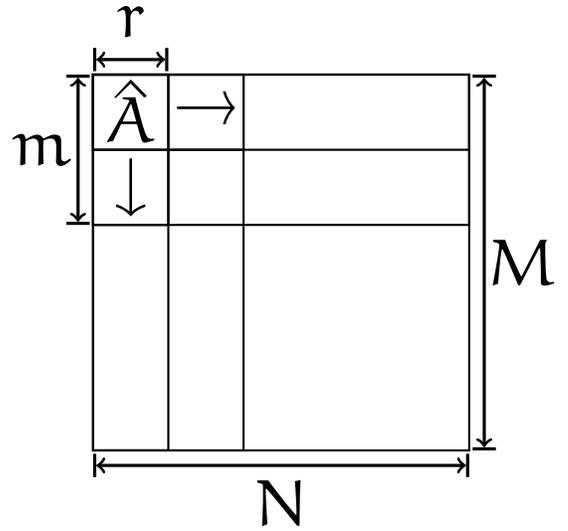
\begin{wrapfigure}{R}{0.5\textwidth}
    \begin{tikzpicture}
	\makeatletter
		\let\cs\@undefined
	\makeatother

	\newlength{\cs} 
	\setlength{\cs}{1.0cm}

	\foreach \x in {1.5\cs}
		\foreach \y in {1\cs}
	{\huge
		\draw [thick] (\x, \y) rectangle +(5\cs, 5\cs);
		\draw [thick] (\x, \y) rectangle +(1\cs, 5\cs);
		\draw [thick] (\x, \y) rectangle +(2\cs, 5\cs);
		\draw [thick] (\x, \y + 3\cs) rectangle +(1\cs, 2\cs);
		\draw [thick] (\x, \y + 4\cs) rectangle +(5\cs, 1\cs);
		\draw [thick] (\x, \y + 3\cs) rectangle +(5\cs, 2\cs);
		\draw [thick] (\x, \y + 4\cs) rectangle +(2\cs, 1\cs);

        \node at (\x + 0.5\cs, \y + 4.5\cs) {\( \hat A \)};
        \node at (\x + 1.5\cs, \y + 4.5\cs) {\( \rightarrow \)};
        \node at (\x + 0.5\cs, \y + 3.5\cs) {\( \downarrow \)};
		
		\node at (\x + 0.5\cs, \y + 5.7\cs) {\( r \)};
		\node at (\x - 0.7\cs, \y + 4\cs) {\( m \)};
		\node at (\x + 5.7\cs, \y + 2.5\cs) {\( M \)};
		\node at (\x + 2.5\cs, \y - 0.7\cs) {\( N \)};
		
		\draw [|<->|, very thick] (\x, \y + 5.2\cs) -- (\x + 1\cs, \y + 5.2\cs);
		\draw [|<->|, very thick] (\x - 0.2\cs, \y + 3\cs) -- (\x - 0.2\cs, \y + 5\cs);
		
		\draw [|<->|, very thick] (\x + 5.2\cs, \y) -- (\x + 5.2\cs, \y + 5\cs);
		\draw [|<->|, very thick] (\x, \y - 0.2\cs) -- (\x + 5\cs, \y - 0.2\cs);
	}
\end{tikzpicture}
    \caption{Fast $CGR$: extending $\hat A = A_{\mathcal{I},\mathcal{J}}$ (obtained by {\myfont maxvol}) by two executions of {\myfont maxvol2}.}
    \label{cgr-fig}
\end{wrapfigure}

Combining all the estimates together, and taking into account the low probability to start from a low volume submatrix, we obtain the following overall complexity if we use Householder-based version of {\myfont maxvol2} (algorithm \ref{hmaxvol2-alg} in the Appendix)
\[
\begin{gathered}
\approx O\left( (M+N)r^2 \cdot IT + \right. \hfill \\
  + \left. (M+N)r \cdot iter_{maxvol} + (M+N)(m+n)r \right) \hfill
\end{gathered}
\]

Let's require the condition with $c = const > 1$  instead of the submatrix dominance  (although even for $c = 1$ the number of the row replacements in {\myfont maxvol} is of order $r$). Then the complexity estimate becomes
\[
\approx O\left( (M+N)r^2 \ln r \cdot IT + (M+N) (m+n)r \right)
\]

This algorithm works faster because all the row and column replacements are produced by the algorithm {\myfont maxvol}, and not by {\myfont Dominant-C} and {\myfont Dominant-R}. 
Also, appending a single row or column is cheaper than replacing. 
Nevertheless, numerical experiments show that the approximation by this algorithm is worse than by {\myfont maxvol-proj}.

Let us shortly mention some other ways to simplify and speed up the algorithm. 
The number of switches between rows and columns can be restricted (high approximation accuracy and large volume are already achieved in step 4). 
The choice $c \gg 1$ can guarantee a small number of replacements (although in practice $c = 1$ works fast enough). 
The criterion for substitutions can be simplified (through the replace of the condition $B_{ij} > 1$ by $|C_{ij}| > 1$ or $(1 + l_i)(1-l_j) > 1$, thus avoiding the calculation of $B$).
The number of replacements can be limited directly (for example, by the number $n$, which also does not greatly aggravate the accuracy). 

In addition, the {\myfont maxvol} algorithm can be replaced by analogs. 
For example, one can obtain a set of rows and columns with just {\myfont pre-maxvol} or the Bebendorf algorithm \cite{Beb}. 
The use of the latter also gives the necessary rank (estimated from above) that guarantees the required accuracy. 

\subsection{Application to the subset selection problem}\label{appsubsec}

Algorithms {\myfont maxvol} and {\myfont Dominant-C} can be used to search for the columns containing a strongly nondegenerate submatrix. 
The obtained estimates of the step number allow to compare them with algorithms from \cite{SubSel}.
By the item 2 of Lemma \ref{first-lem}, choice of the submatrix $\hat A \in \mathC^{r \times n}$ in the rows $R \in \mathC^{r \times N}$ guarantees the inequality
\[
  \|\hat A^+ R\|_2 \leqslant \sqrt{1 + \frac{r(N-n)}{n-r+1}}.
\]
It immediately follows that
\[
  \sigma_k(\hat A^+) \leqslant \sigma_k(R^+) \sqrt{1 + \frac{r(N-n)}{n-r+1}}.
\]
For the Frobenius norm, using another expression from item 2 of the Lemma \ref{first-lem}, we obtain 
\begin{equation}\label{AR_F2}
  \|\hat A^+\|_F \leqslant \|R^+\|_2 \sqrt{r + \frac{r(N-n)}{n-r+1}}.
\end{equation}

The table \ref{tabl_mat2} contains the estimates of $\|\hat A^+\|_2$ and $\|\hat A^+\|_F$ for {\myfont maxvol} and {\myfont Dominant-C}, and the estimates for the algorithms from \cite{SubSel}. 
Note that the algorithm $GEQR$ \cite{rrqr}, equivalent to {\myfont maxvol} in case of $k=r$ columns, has already been mentioned in \cite{SubSel}.  
However, we proved a better estimate on the step number.

The algorithms are sorted by the number of operations. 

In order to simplify the comparison for {\myfont Dominant-C}, we give the case $n = 2r-1$ and $c = 2$.
It is assumed that the {\myfont pre-maxvol} algorithm is used before {\myfont maxvol} or {\myfont Dominant-C}, thus we can estimate the number of operations.

\begin{table}[ht]
\caption{Methods for finding a strongly nondegenerate rectangular submatrix $\hat A \in \mathC^{r \times n}$ in the rows $R \in \mathC^{r \times N}$.}\label{tabl_mat2}
\begin{center}
\small
\begin{tabular}{|c|c|c|c|}
\hline
Method & $\|\hat A^+\|_F^2 / \|R^+\|_F^2$ & $\|\hat A^+\|_2^2 / \|R^+\|_2^2$ & Complexity \\
\hline
Theorem 3.7 ($\delta = 1/2$) from & \multirow{2}{*}{$4N$} & \multirow{2}{*}{$4N$} & \multirow{2}{*}{$O(N r^2 + n \log n)$} \\
\cite{SubSel}, $n \geqslant 32k \ln(4k)$ & & & \\
\hline
{\myfont maxvol} & \multirow{2}{*}{$\left(1 + c(N-r)\right)\frac{r \|R^+\|_2}{\|R^+\|_F}$} & \multirow{2}{*}{$1 + cr(N-r)$} & \multirow{2}{*}{$O(Nr^2 \log r / \log c)$} \\
\cite{maxvol}, $n = r$ & & & \\
\hline
{\myfont Dominant-C} & \multirow{2}{*}{$\left(2\frac{N+1}{r}-4\right)\frac{r \|R^+\|_2}{\|R^+\|_F}$} & \multirow{2}{*}{$2(N-2r)+3$} & \multirow{2}{*}{$O(N r^2 \log r)$} \\
$n = 2r-1$, $c = 2$ & & & \\
\hline
{\myfont Dominant-C} & \multirow{2}{*}{$\left(1 + \frac{N-n}{n-r+1}\right)\frac{r \|R^+\|_2}{\|R^+\|_F}$} & \multirow{2}{*}{$1 + r\frac{N-n}{n-r+1}$} & \multirow{2}{*}{$O(Nnr \log n / \log c)$} \\
$n \geqslant r$ & & & \\
\hline
Theorem 3.11 ($\delta = 1/2$) from & \multirow{2}{*}{$c(N-r+1)$} & \multirow{2}{*}{$cr(N-r+1)$} & \multirow{2}{*}{$O(N r^3 / \log c)$} \\
\cite{SubSel}, $n = r$ & & & \\
\hline
Theorem 3.5 from & \multirow{2}{*}{$\frac{(1 + \sqrt{\frac{N}{n}})^2}{(1 - \sqrt{\frac{r}{n}})^2}$} & \multirow{2}{*}{$\frac{(1 + \sqrt{\frac{N}{n}})^2}{(1 - \sqrt{\frac{r}{n}})^2}$} & \multirow{2}{*}{$O(N n r^2)$} \\
\cite{SubSel}, $n > r$ & & & \\
\hline
Theorem 3.1 from & \multirow{2}{*}{$\frac{N-r+1}{n-r+1}$} & \multirow{2}{*}{$r \frac{N-r+1}{n-r+1}$} & \multirow{2}{*}{$O(N r^2 + N(N-n)r)$} \\
\cite{SubSel}, $n \geqslant r$ & & & \\
\hline
Cons. 3.3 from & \multirow{2}{*}{$\frac{N-r+1}{n-r+1} \cdot \frac{r \|R^+\|_2}{\|R^+\|_F}$} & \multirow{2}{*}{$1 + r\frac{N-n}{n-r+1}$} & \multirow{2}{*}{$O(N r^2 + N(N-n)r)$} \\
\cite{SubSel}, $n \geqslant r$ & & & \\
\hline
\end{tabular}
\end{center}
\end{table}

The last two algorithms are based on column removals. 
Since their complexities depend on $N$ quadratically, the application is possible only for small matrices. 

Only the algorithms which are at least about $r$ times slower give better estimates of $\|\hat A^+\|_2$ and $\|\hat A^+\|_F$ than {\myfont maxvol} and {\myfont Dominant-C}. 
For the Frobenius norm, {\myfont maxvol} and {\myfont Dominant-C} have worse guarantees only if the ratio $\frac{r \|R^+\|_2}{\|R^+\|_F}$ is not close to 1. 
In addition, if {\myfont maxvol} or {\myfont Dominant-C} selects a submatrix with almost equal singular values, then the factor $r$ in the 2-norm estimate disappears (see (\ref{AR_F2})).

When comparing the complexity, we see that the only algorithm from \cite{SubSel} that can compete with {\myfont Dominant-C} in speed, leads to a larger error and requires significantly more columns, while its number of operations is only $O(\log r)$ times smaller.

\section{Numerical experiments}\label{exp-sec}

In this section, we study the efficiency of the above algorithms on the cross approximation problem.

In matrix $A \in \mathbb{C}^{M \times N}$ we select $n \geq r$ columns $C \in \mathbb{C}^{M \times n}$ and $m \geq r$ rows $R \in \mathbb{C}^{m \times N}$, and the cross approximation  has the form $CGR,$ with the {\it approximation generator} $G \in \mathbb{C}^{n \times m},$ of rank at most $r$.

Generally, the errors for cross approximations are calculated in the 2-norm and the Frobenius norm. 
The existing lower bounds \cite{lowerCCA} do not allow to guarantee the high approximation accuracy in the 2-norm. 
For the Frobenius norm, the result has been recently obtained with the same coefficient as for the $C$-norm in \cite{me}. Unfortunately, it does not apply to the maximum volume submatrices.

\begin{theorem}[\cite{dan2}]
For any matrix $A \in \mathbb{C}^{M \times N}$ there exists a skeleton approximation $C \hat {A }^{-1} R$ such that
\begin{equation}\label{eq:main}
\|A - C \hat{A}^{-1} R\|_F \leq \left(r + 1\right) \|A - A_r\|_F
\end{equation}
\end{theorem}

Still, we hope that this estimate and its analog for the projective volume are often achieved on the matrices obtained by the algorithms maximizing the volume and the projective volume. This is the hypothesis, which we will try to check numerically.

The rest of this section is split into subsections, where we make a statement about the quality of the approximations and show the corresponding numerical results.

\subsection{Projective volume can lead to approximations, arbitrary close to SVD}

The complexity of {\myfont maxvol-proj} is linear in matrix size, while SVD is cubic, so the ability to replace SVD by {\myfont maxvol-proj} with is a small multiplicative error can be very useful.

Let us represent the cross approximation error in the form
\begin{equation}\label{eps-eq}
  \|A - CGR\|_F = (1 + \varepsilon) \|A - A_r\|_F.
\end{equation}
The question is, how many rows and columns and how much time it takes to achieve (\ref{eps-eq})?

We will show, that {\myfont maxvol-proj} can produce approximations with the error $(1+\varepsilon) \|A - A_r\|_F$ in $O(Nr^2 / \varepsilon^2)$ time. Fast $CGR$-approximation construction (algorithm \ref{fastcgr-alg}) with Householder-based {\myfont maxvol2} (algorithm \ref{hmaxvol2-alg} in the Appendix) has complexity $O(Nr^2 / \varepsilon)$, but should be used with caution, because {\myfont maxvol2} provides fewer guarantees than {\myfont maxvol-rect}. Nevertheless, we will see that it still works almost as well as {\myfont maxvol-proj}.

We hypothesize the expected ratio of the obtained error to the best approximation error (in Frobenius norm) to be the same as for the case of the $C$-norm in \cite{me}
\begin{equation}\label{hyp-eq}
  {\mathbb{E}_{U,V}} \| A - C \hat A_r^+ R \|_F \leqslant \sqrt{\left( 1 + \frac{r}{m-r+1} \right) \left( 1 + \frac{r}{n-r+1} \right)} \| A - A_r \|_F
\end{equation}
in the same notations as before. Hereinafter $U$ and $V$ in the expectation denote random unitary matrices. They are also left and right singular vectors of $A$. We hope that by increasing $n$ and $m$ we can reach arbitrary small $\varepsilon$ and that $1/\varepsilon$ depends linearly on the size of the submatrix. Hereinafter we limit the number of swaps by $2n = 2m$, so there is no logariphmic factor in complexity.

It is important to note that all the algorithms in this paper do not observe the entire matrix and thus there are no guarantees for the worst case.

Figure \ref{new-fig} shows that the value $1/\varepsilon$ indeed depends linearly on the submatrix size on the average, and even is larger than expected ($\varepsilon \approx \frac{r}{n-r+1}$ in (\ref{hyp-eq})). However, there are still cases when the error coefficient $1 + \varepsilon$ is larger than $1 + \frac{r}{n-r+1}$ (when $1/\varepsilon$ is smaller than $\frac{n-r+1}{r}$), which is also expected. Nevertheless, they appear only in the case, when singular values of the matrix quickly decrease. We discuss how to significantly decrease the error in this case at the end of this section. Finally, fast $CGR$ (2 maxvol2) also shows linear dependence, which means it often successfully finds large projective volume submatrices, in spite of fewer guarantees.

\begin{figure}[ht]
\begin{subfigure}[b]{0.49\textwidth}
\centering
\includegraphics[width=\columnwidth]{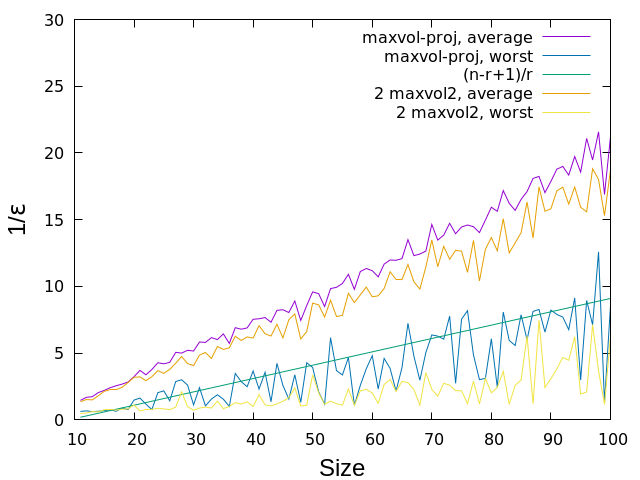}
\caption{The $k$-th singular value of the matrix equals $1/2^k$.}
\end{subfigure}
\begin{subfigure}[b]{0.49\textwidth}
\centering
\includegraphics[width=\columnwidth]{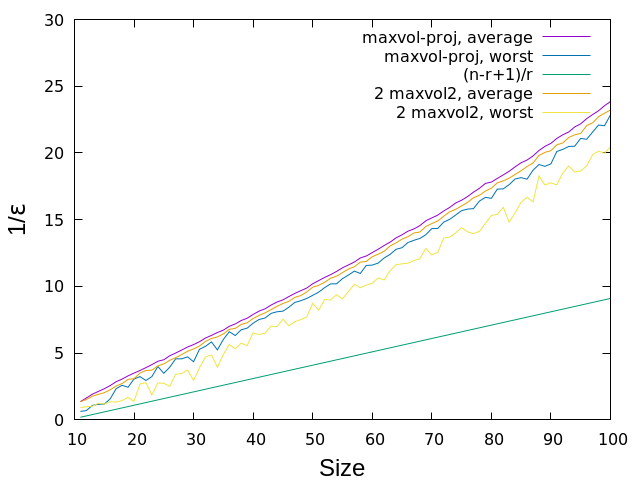}
\caption{First $r$ singular values are equal to $20$, the rest are $1$.}
\end{subfigure}
\caption{
Dependence of $1/\varepsilon$ (\ref{eps-eq}) on the number of rows and columns in the submatrix, used in different projective volume search algorithms.
Matrix size $400 \times 400$, rank $r = 10$, matrix singular values are described under the figures.
The results are averaged over $100$ random matrix generations. Note that $1/\varepsilon$ is plotted, so larger values mean smaller error. $1/\varepsilon$ for largest error (``worst'' case) among these $100$ random matrices is also plotted.}
\label{new-fig}
\end{figure}

As we see, the hypothetical estimate (\ref{hyp-eq}) is rather rough for the average case.
However, it is possible to construct an estimate much closer to reality, but less justified.

\subsection{Average error estimates}

Hereinafter in numerical experiments we consider 3 cases: $m = n = r$, $m = 2r, n = r$ and $m = n = 2r$, denoted on the figures by $r * r$, $2r * r$ and $2r * 2r$.
According to our hypothesis, the corresponding errors are estimated as 
\begin{eqnarray}
  {\mathbb{E}_{U,V}} \| A - C \hat A^{-1} R \|_F & \leqslant & (r+1) \| A - A_r \|_F, \nonumber \\
  {\mathbb{E}_{U,V}} \| A - C \hat A^+ R \|_F & \leqslant & \sqrt{2r+1} \| A - A_r \|_F, \nonumber \\
  {\mathbb{E}_{U,V}} \| A - C \hat A_r^+ R \|_F & \leqslant & \left(1 + \frac{r}{r+1} \right) \| A - A_r \|_F. \nonumber
\end{eqnarray}

By item 2 of Lemma \ref{first-lem}, the coefficients like $\sqrt{\frac{n+1}{n-r+1}}$ come from the the Frobenius norm of the pseudoinverse of some submatrix $\hat U \in \mathbb{C}^{n \times r}$. 
For the 2-norm estimates, we use the 2-norm of the pseudoinverse. 
Using the equations from item 2 of the Lemma \ref{first-lem} to replace  $\sqrt{\frac{n + 1}{n - r + 1}}$, we obtain the following estimate:
\begin{equation}\label{advappr}
  {\mathbb{E}_{U,V}} \| A - C \hat A_r^+ R \|_F \approx {\mathbb{E}_{U',V'}} \sqrt{\left( 1 + \frac{1}{M-r} (\| (\hat U')^+ \|_F^2 - r) \right)\left( 1 + (\frac{1}{N-r} \| (\hat V')^+ \|_F^2 - r) \right)} \| A - A_r \|_F
\end{equation}
$\hat U' \in \mathC^{m \times r}$ and $\hat V' \in \mathC^{r \times n}$ here are the dominant submatrices of some random unitary matrices $U' \in \mathC^{M \times r}$ and $V' \in \mathC^{r \times N}$. 
Their sizes for each of the three cases correspond to the sizes of $\hat A$. For  $M = N = 100$, the values of the coefficients are shown in figure \ref{u2-pic}.

\begin{SCfigure}
\centering
\includegraphics[width=0.7\columnwidth]{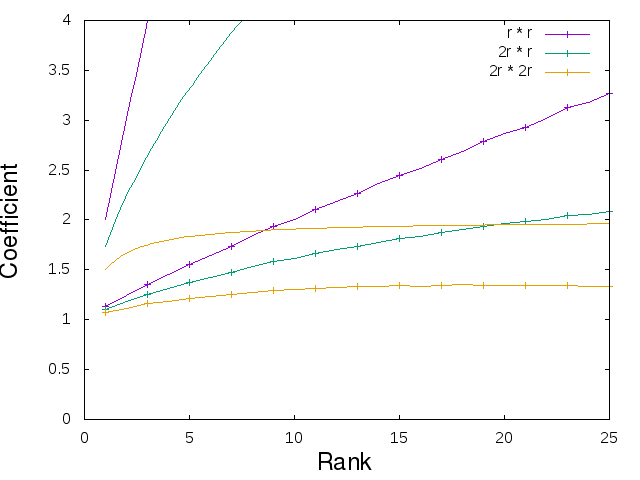}
\caption{
The coefficients before the Frobenius norm error in (\ref{advappr}) presented as a function of rank for various submatrix sizes (lines with crosses).
The corresponding estimates from (\ref{hyp-eq}) (lines without crosses). 
The sizes of unitary matrices are $100 \times r$.
The results are averaged over $100$ random matrix generations.
}
\label{u2-pic}
\end{SCfigure}

As one can see, the difference is significant. 
The following figures do not have an upper bound to avoid chaos.

Let's compare our estimates to practice.
In the figure \ref{decpic}, the submatrices are constructed by {\myfont maxvol} \cite{maxvol}, {\myfont maxvol-rect} and {\myfont maxvol-proj}.

\begin{figure}[ht]
\begin{subfigure}[b]{0.49\textwidth}
\centering
\includegraphics[width=\columnwidth]{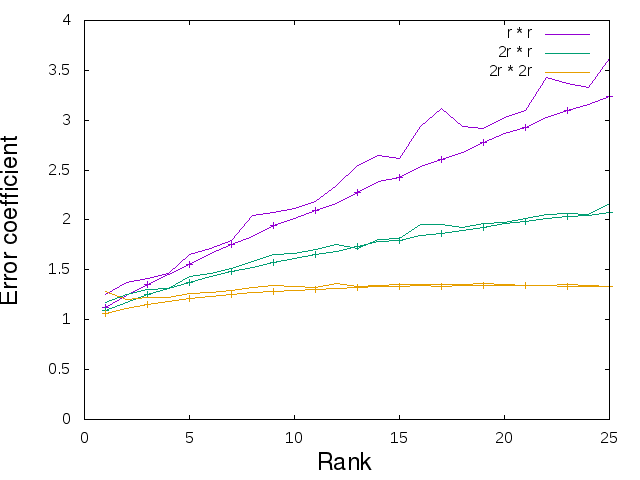}
\caption{The $k$-th singular value of the matrix equals $1/2^k$.}
\end{subfigure}
\begin{subfigure}[b]{0.49\textwidth}
\centering
\includegraphics[width=\columnwidth]{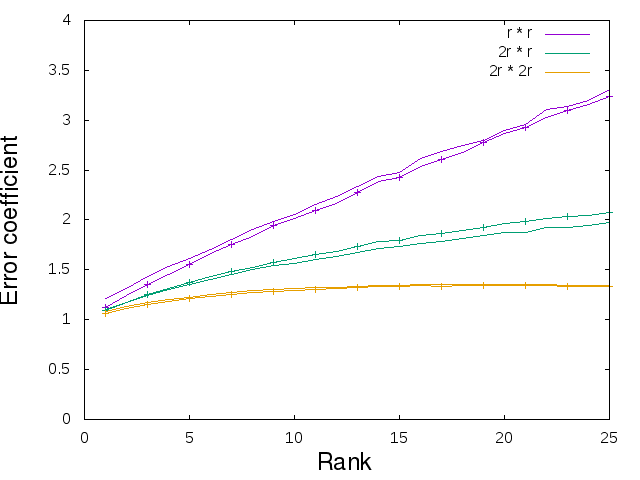}
\caption{First $r$ singular values are equal to $10$, the rest are $1$.}
\end{subfigure}
\caption{
The ratio of the cross approximation error to the best Frobenius norm error presented as a function of rank for various submatrix sizes (lines without crosses).
Estimates from (\ref{advappr}) (lines with crosses). 
Matrix size $100 \times 100$, matrix singular values are described under the figures.
The results are averaged over $100$ random matrix generations.}
\label{decpic}
\end{figure}

The upper bound estimate always holds, and the estimate (\ref{advappr}) is reasonably accurate.
We note that the relative error is less in the case of equal singular values of the error. 
The approximation is accurate even though in both cases the Frobenius norm of the error is comparable with the minimum singular value of the approximation. Note that the average error coefficient of the algorithms does not depend on the singular values of the error. The quality of the approximation is the same, whether the singular values of $A-A_r$ are small and quickly decrease, or they are large (compared to the singular values of $A_r$) and do not decrease after $r$-th at all.

If the ratio of singular values becomes about $1/1000$, the graphs visually coincide. 

Of course, the average case scenario is not too informative and we should also look at how bad the error can get and how often.

\subsection{Larger size means smaller variance}

Here we show that using more rows and columns also significantly decreases the variance and thus large errors become much rarer. To do it, we plot the distributions of the cross approximation error.

\begin{figure}[ht]
\begin{subfigure}[b]{0.49\textwidth}
\centering
\includegraphics[width=\columnwidth]{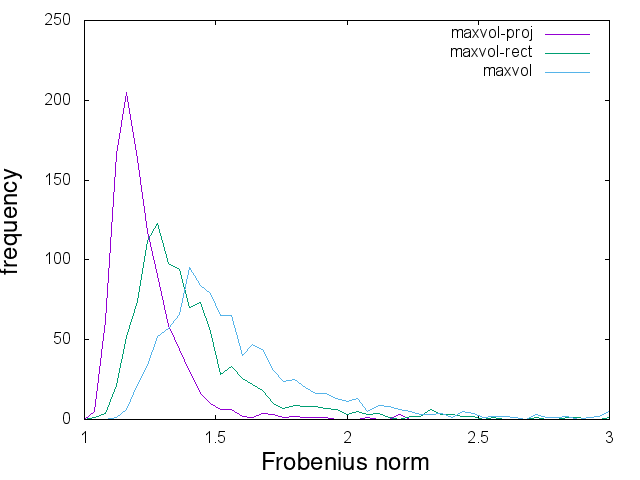}
\caption{Rank 5.}
\end{subfigure}
\begin{subfigure}[b]{0.49\textwidth}
\centering
\includegraphics[width=\columnwidth]{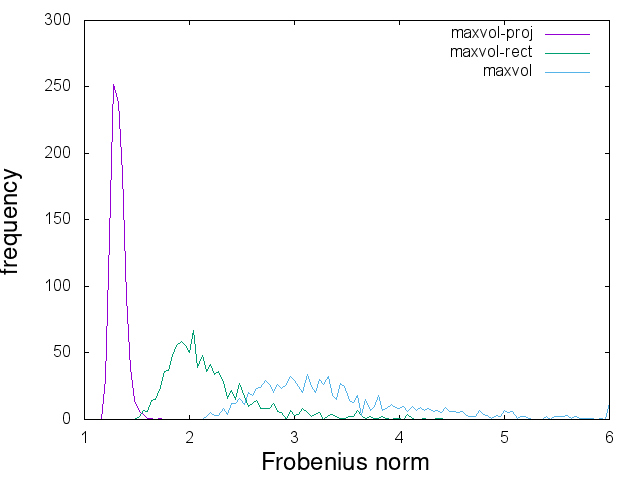}
\caption{Rank 25.}
\end{subfigure}
\caption{
Distributions of the Frobenius norms of the error for different cross approximation algorithms.
Matrix size $100 \times 100$, different values of rank $r$,
{\myfont maxvol-rect} uses $2r$ rows. 
The $k$-th singular value equals $1/2^k$.
The distribution is based on  $1000$ random matrix generations.}
\label{fhistpic}
\end{figure}

\begin{figure}[ht]
\begin{subfigure}[b]{0.49\textwidth}
\centering
\includegraphics[width=\columnwidth]{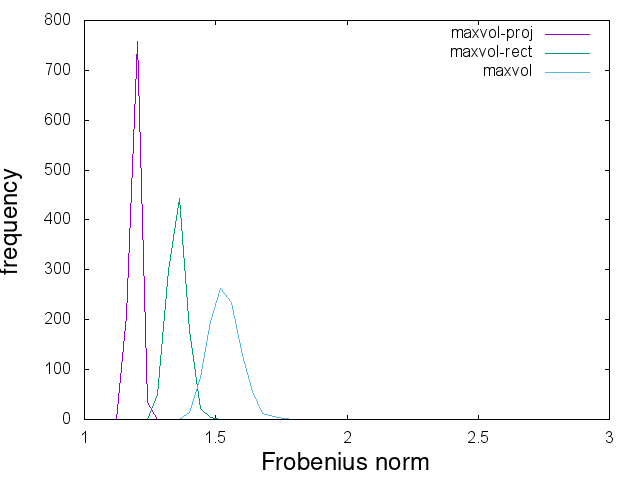}
\caption{Rank 5.}
\end{subfigure}
\begin{subfigure}[b]{0.49\textwidth}
\centering
\includegraphics[width=\columnwidth]{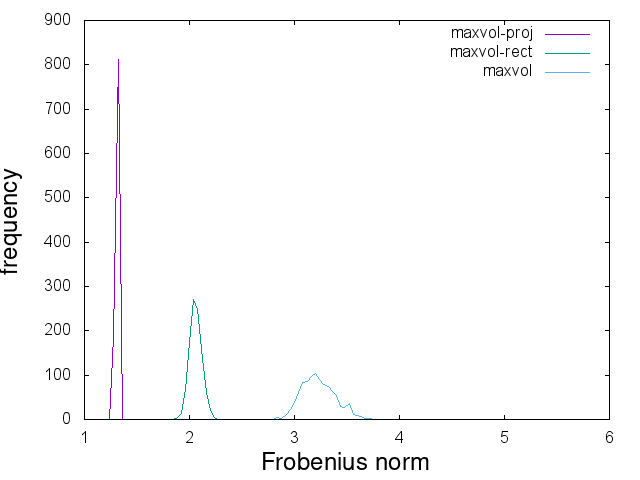}
\caption{Rank 25.}
\end{subfigure}
\caption{Distributions of the Frobenius norms of the error for different cross approximation algorithms.
Matrix size $100 \times 100$, different values of rank $r$,
{\myfont maxvol-rect} uses $2r$ rows. 
The singular values of the error are $1000$ times less than the singular values of the best approximation.
The distribution is based on $1000$ random matrix generations.
}
\label{shistpic}
\end{figure}

The error distribution histograms are shown in figures \ref{fhistpic} and \ref{shistpic}. 
It is indeed easy to see that not only the average error but also the error variance is less for {\myfont maxvol-proj}. 

\begin{figure}[ht]
\begin{subfigure}[b]{0.49\textwidth}
\centering
\includegraphics[width=\columnwidth]{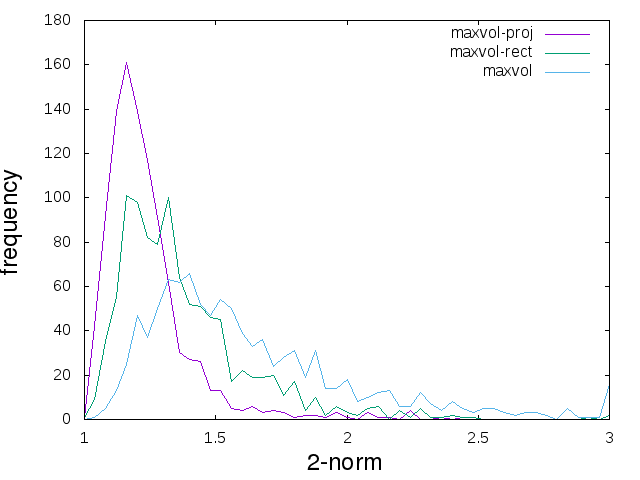}
\caption{Rank 5.}
\end{subfigure}
\begin{subfigure}[b]{0.49\textwidth}
\centering
\includegraphics[width=\columnwidth]{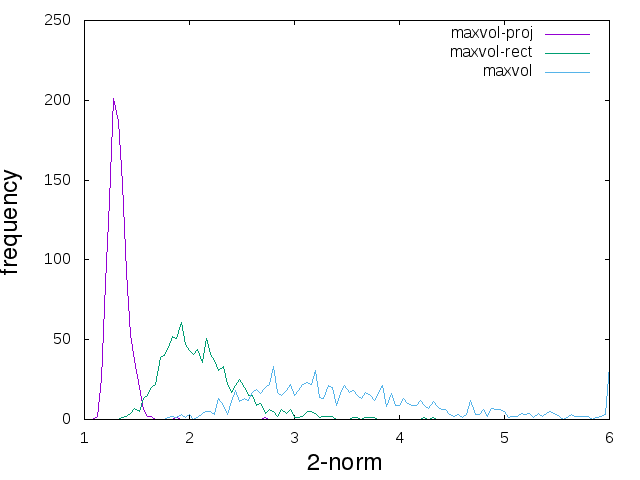}
\caption{Rank 25.}
\end{subfigure}
\caption{
Distributions of the error 2-norm for different cross approximation algorithms.
Matrix size $100 \times 100$, different values of rank $r$,
{\myfont maxvol-rect} uses $2r$ rows. 
The $k$-th singular value equals $1/2^k$.
The distribution is based on $1000$ random matrix generations.
}
\label{fhistpic2}
\end{figure}

\begin{figure}[ht]
\begin{subfigure}[b]{0.49\textwidth}
\centering
\includegraphics[width=\columnwidth]{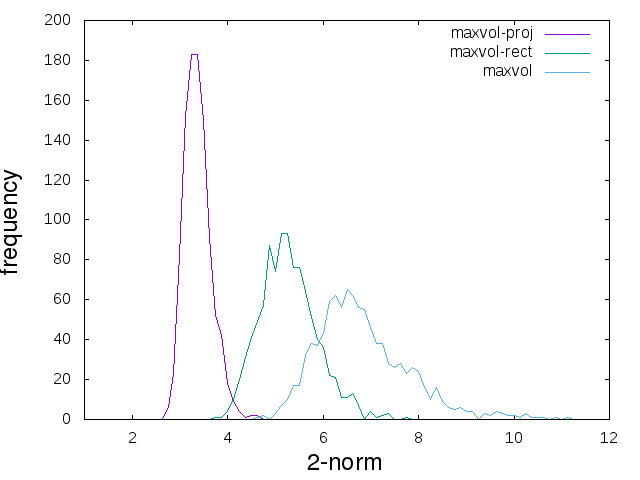}
\caption{Rank 5.}
\end{subfigure}
\begin{subfigure}[b]{0.49\textwidth}
\centering
\includegraphics[width=\columnwidth]{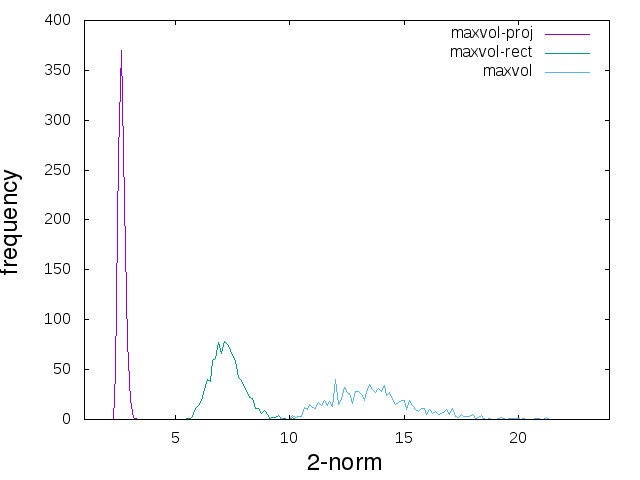}
\caption{Rank 25.}
\end{subfigure}
\caption{
Distributions of the error 2-norm for different cross approximation algorithms.
Matrix size $100 \times 100$, different values of rank $r$,
{\myfont maxvol-rect} uses $2r$ rows. 
The singular values of the error are $1000$ times less than the singular values of the best approximation.
The distribution is based on $1000$ random matrix generations.
}
\label{shistpic2}
\end{figure}

Figures \ref{fhistpic2} and \ref{shistpic2} show the similar histograms for the 2-norm. 
If the 2-norm and the Frobenius norm of the error are close, the $CGR$ approximation also has a small error. And in the case of a large difference between the error norms the approximation is much less accurate.

The $C$-norm of the error was studied in detail in \cite{me}.
Figure \ref{cnormhist} shows that the distributions of the $C$-norm and of the Frobenius norm are close, which coincides with our hypothesis that the coefficients should be similar.

\begin{wrapfigure}{R}{0.5\textwidth}
\includegraphics[width=0.5\columnwidth]{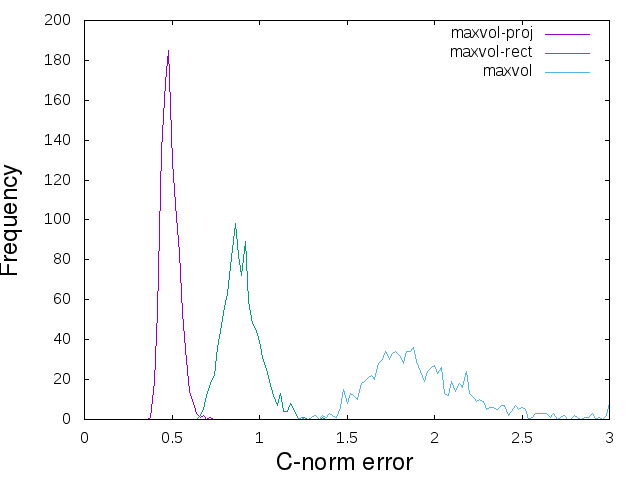}
\caption{
Distributions of the $C$-norm error for different cross approximation algorithms.
Matrix size $100 \times 100$, approximation rank $r=10$, $n = 20$.
First $10$ singular values are equal to $10$, the rest are $1$.
The distribution is based on $2000$ random matrix generations.
}
\label{cnormhist}
\end{wrapfigure}

It is also noticeable that in the case of methods with $r$-projective volume maximization, the error does not heavily depend on the particular maximization approach. 
As already mentioned, instead of {\myfont maxvol-proj}, one can quite quickly apply {\myfont maxvol2} twice to the found submatrix of size $r \times r$.

\subsection{Comparison of different projective volume maximization methods}

In this subsection, we compare different algorithms for large projective volume search to see if the exact algorithm of choosing more rows and columns is important or we can use any intuitively reasonable approach.

In addition to already presented algorithms, we examine two more methods that work asymptotically slower:
\begin{enumerate}
\item maxvol of rank 2r: 
The {\myfont maxvol} algorithm is used for the size $2r \times 2r$, and then the resulting submatrix is $r$-pseudoinversed;
\item maxvol2r: instead of the second applying of {\myfont maxvol2}, the matrix is expanded by the column $c$ containing the maximum value of $\| \hat A_r^+ c \|_2 $ (the efficiency justification can be found in \cite{me}). 
\end{enumerate}

Note that instead of the exact singular value decomposition of a submatrix, we can use an approximate technique (for example, $RRQR$).
Often (but not always!) the error does not increase much.

\begin{SCfigure}
\centering
\includegraphics[width=0.7\columnwidth]{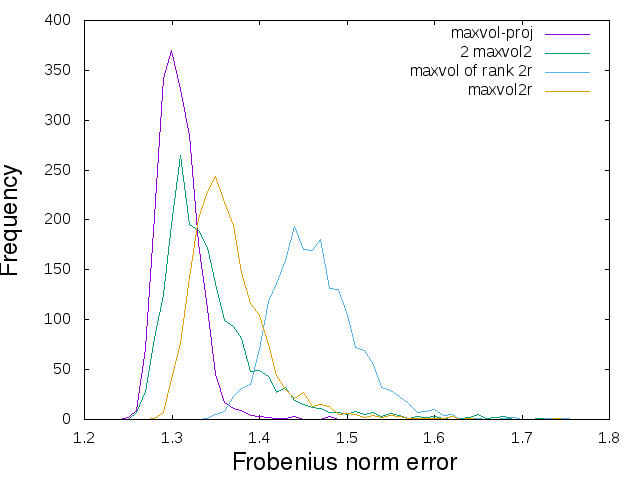}
\caption{
Distributions of the Frobenius norms of the error for different methods of $r$-projective volume maximization.
Matrix size $100 \times 100$, approximation rank $r=10$, $n = 20$.
First $r$ singular values are equal to $10$, the rest are $1$.
The distribution is based on $2000$ random matrix generations.
}\label{dif_max_gr}
\end{SCfigure}

As figure \ref{dif_max_gr} shows, {\myfont maxvol-proj} yields the best mean error and the best variance, although the difference is not too large. 

\subsection{Truncated SVD can improve the approximation in particular cases.}

Previously we have seen that the coefficient for the approximation error does not depend on the singular values of $A_r$. It must be noted, however, that in the case, when singular values decrease very quickly, the approximation quality can be further improved.

Let us consider a non-random matrix example. 
Let
\begin{equation}\label{tab_mat}
  (A_1)_{i,j} = (i^{1/3}+j^{1/3})^2 \sqrt{\frac{1}{i} + \frac{1}{j}}; \quad i,j = \overline{1,n}.
\end{equation}
It is the so-called ``ballistic core'' used in the coagulation and fragmentation problems \cite{coag}. 
Results for several matrix sizes are presented in the table \ref{tab:t1}. 
The approximation accuracy is of the order $10^{-5} - 10^{-6}$.
As always, {\myfont maxvol-rect} uses $2r$ rows. 
The starting columns and rows are selected randomly.

\begin{table}[ht]
\caption{The cross approximation errors for the matrix $A_1$ (\ref{tab_mat}).}\label{tab:t1}
\begin{center}
\small
\begin{tabular}{|c|c|c|c|c|c|}
\hline
Size & \multirow{2}{*}{SVD} & \multirow{2}{*}{{\myfont maxvol}} & \multirow{2}{*}{{\myfont maxvol-rect}} & \multirow{2}{*}{{\myfont maxvol-proj}} & $TSVD$ after {\myfont maxvol}\\
and rank &  &  &  &  & of rank $r+2$ \\
\hline
$n = 800$, $r = 12$ & $1.01 \cdot 10^{-5}$ & $5.40 \cdot 10^{-5}$ & $5.15 \cdot 10^{-5}$ & $3.23 \cdot 10^{-5}$ & $1.02 \cdot 10^{-5}$ \\
\hline
$n = 400$, $r = 11$ & $6.09 \cdot 10^{-6}$ & $2.64 \cdot 10^{-5}$ & $2.25 \cdot 10^{-5}$ & $1.59 \cdot 10^{-5}$ & $6.13 \cdot 10^{-6}$ \\
\hline
$n = 200$, $r = 10$ & $3.59 \cdot 10^{-6}$ & $1.23 \cdot 10^{-5}$ & $1.04 \cdot 10^{-5}$ & $7.09 \cdot 10^{-6}$ & $3.59 \cdot 10^{-6}$ \\
\hline
$n = 100$, $r = 9$ & $2.01 \cdot 10^{-6}$ & $5.41 \cdot 10^{-6}$ & $4.87 \cdot 10^{-6}$ & $3.35 \cdot 10^{-6}$ & $2.01 \cdot 10^{-6}$ \\
\hline
\end{tabular}
\end{center}
\end{table}

In the last column we use truncated singular value decomposition ($TSVD$) of the $C \hat A^{-1} R$ approximation of $A_1$ with rank $r+2$ (which takes $O(Nr^2)$), so that
\[
  A_1 \approx \left(C \hat A^{-1} R\right)_r, \quad \hat A \in \mathbb{R}^{(r+2) \times (r+2)}.
\]
We see that with this approach
\[
  \|A_1 - \left(C \hat A^{-1} R\right)\|_F \leqslant 1.01 \|A_1 - (A_1)_r\|_F,
\]
though we used only $r+2$ rows and columns. The reason for this fact lies in the fast decrease of singular values of $A_1$.

Suppose we want to construct some rank-$r$ approximation of $A$ with another arbitrary matrix $\tilde A$. Let $P_r$ be an orthogonal projection on the first left singular vectors of $A$, so that $P_r A = A_r$. Then
\[
\begin{gathered}
  {\left\| {A - {{\left( {\tilde A} \right)}_r}} \right\|_F} \leqslant {\left\| {A - \tilde A} \right\|_F} + {\left\| {\tilde A - {{\left( {\tilde A} \right)}_r}} \right\|_F} \\ 
   \leqslant {\left\| {A - \tilde A} \right\|_F} + {\left\| {\left( {I - {P_r}} \right)\tilde A} \right\|_F} \\ 
   \leqslant {\left\| {A - \tilde A} \right\|_F} + {\left\| {\left( {I - {P_r}} \right)A} \right\|_F} + {\left\| {\left( {I - {P_r}} \right)\left( {A - \tilde A} \right)} \right\|_F} \\ 
   \leqslant 2{\left\| {A - \tilde A} \right\|_F} + {\left\| {A - {A_r}} \right\|_F}. \\ 
\end{gathered}
\]
For $\tilde A = CGR$ with rank $r + k$ approximation error $\|A - \tilde A\|_F \tilde \|A - A_{r+k}\|_F$ is much smaller than $\|A - A_r\|_F$ when singular values quickly decrease, so we get the error close to $\|A - A_r\|_F$ with only a few additional rows and columns.

Nevertheless, for the closely distributed singular values, the truncated $SVD$ of the constructed approximation no longer gives any advantages, and {\myfont maxvol-proj} becomes more efficient. 
For such an example, we replace the singular values of the error of $A_1$ by equal numbers.
That is, consider a matrix 
\begin{equation}\label{tab_mat2}
{A_2} = {U_1}{\Sigma _2}{V_1},
\end{equation}
such that
\[
{A_1} = {U_1}{\Sigma _1}{V_1}
\]
and
\[
\begin{gathered}
  {\sigma _i}\left( {{\Sigma _2}} \right) = {\sigma _i}\left( {{\Sigma _1}} \right),\quad i = \overline {1,r} , \hfill \\
  {\sigma _i}\left( {{\Sigma _2}} \right) = {\sigma _j}\left( {{\Sigma _2}} \right),\quad i,j > r, \hfill \\
  {\left\| {({A_1}) - {{({A_1})}_r}} \right\|_F} = {\left\| {({A_2}) - {{({A_2})}_r}} \right\|_F}. \hfill \\ 
\end{gathered} 
\]
Thus, $SVD$-based approximation gives the same error for $A_1$ and $A_2$.

\begin{table}[ht]
\caption{The cross approximation errors for the matrix $A_2$ (\ref{tab_mat2}).}\label{tab:t2}
\begin{center}
\small
\begin{tabular}{|c|c|c|c|c|c|}
\hline
Size & \multirow{2}{*}{SVD} & \multirow{2}{*}{{\myfont maxvol}} & \multirow{2}{*}{{\myfont maxvol-rect}} & \multirow{2}{*}{{\myfont maxvol-proj}} & $TSVD$ after {\myfont maxvol}\\
and rank &  &  &  &  & of rank $2r$ \\
\hline
$n = 800$, $r = 12$ & $1.01 \cdot 10^{-5}$ & $2.02 \cdot 10^{-5}$ & $1.71 \cdot 10^{-5}$ & $1.44 \cdot 10^{-5}$ & $1.59 \cdot 10^{-5}$ \\
\hline
$n = 400$, $r = 11$ & $6.09 \cdot 10^{-6}$ & $1.19 \cdot 10^{-5}$ & $9.63 \cdot 10^{-6}$ & $9.01 \cdot 10^{-6}$ & $9.94 \cdot 10^{-6}$ \\
\hline
$n = 200$, $r = 10$ & $3.59 \cdot 10^{-6}$ & $6.86 \cdot 10^{-6}$ & $6.03 \cdot 10^{-6}$ & $5.03 \cdot 10^{-6}$ & $5.57 \cdot 10^{-6}$ \\
\hline
$n = 100$, $r = 9$ & $2.01 \cdot 10^{-6}$ & $3.84 \cdot 10^{-6}$ & $3.30 \cdot 10^{-6}$ & $2.71 \cdot 10^{-6}$ & $3.11 \cdot 10^{-6}$ \\
\hline
\end{tabular}
\end{center}
\end{table}

Table \ref{tab:t2} shows how the algorithms approximate $A_2$.
One can see that the best results are given by {\myfont maxvol-proj}. 
Truncated $SVD$ after {\myfont maxvol} of higher rank $n = 2r$ is much less useful than for $A_1$. 
In this case simple {\myfont maxvol} also performs better. 
However, its error still grows with size and rank. 
For {\myfont maxvol-proj}, the error never exceeds the $SVD$-based error more than $1.5$ times.

It is also worth noting that the use of Truncated $SVD$ still improves the accuracy: the direct application of {\myfont maxvol} of rank $2r$ gives a more significant error. 
Thus, for the slowly decaying singular values, the use of {\myfont maxvol} of higher rank may not yield any advantages.
To improve accuracy, one must reduce the rank of the approximation. 

If one doesn't know in advance that singular values quickly decrease, it can still be checked in the constructed approximation, and the rank can be reduced the same way after {\myfont maxvol-proj}.

\section{Conclusion}

In general, low-rank approximations based on the maximum volume principle require very few operations and show high accuracy of the approximation. 

When we need to choose between the speed and the accuracy, we can use the algorithm {\myfont maxvol-proj}, its version without {\myfont Dominant-R}, or its simplifications. 
At the same time, the necessary rank can be estimated from above, starting with the Bebendorf algorithm \cite{Beb}. 
If necessary, one can reduce it by applying Truncated $SVD$ to the constructed cross approximation. 
In the case of a fast singular values decay, Truncated $SVD$ can be applied to an approximation of a larger rank, and thus significantly improve the accuracy (compared to the direct rank $r$ approximations). 
Sometimes, this makes it possible to achieve better accuracy without {\myfont maxvol-proj}. 
However, {\myfont maxvol-proj} is a universal method in the sense that the accuracy in the Frobenius norm does not depend (on average) on the distribution of singular values. 

The numerical experiments show that averaging over random matrices of left and right singular vectors gives the Frobenius norm error with the coefficient of the same order as in the $C$-norm. 
This fact is directly related to the Frobenius norm of pseudoinverse to some submatrices of unitary matrices. 
Moreover, the hypothesized coefficient upper bound (\ref{hyp-eq}) is similar to the best known upper bounds \cite{bestCUR} and with the lower bounds both for $n = r$ and for $n \to \infty$ \cite{CWexist, LowerBounds}.

\section{Appendix}
\appendix

\section{Full versions of algorithms}

Below we present the detailed algorithm descriptions and the derivations for some new update formulas. 
The justifications of the update formulas for {\myfont maxvol} \cite{maxvol} and $GEQR$ \cite{rrqr} are in the corresponding articles. References are given in the algorithm headings.
The update formula for {\myfont Dominant-C} is provided in the Algorithms section, Lemma \ref{bij-lem}.

We tried to formulate the algorithms to make the coding in any programming language as close as possible to simple copying. 

\paragraph*{Algorithm {\myfont maxvol} \cite{maxvol}}
\begin{algorithmic}[1]
\REQUIRE{Matrix $A \in \mathbb{C}^{M \times N}$, the starting sets of row indices $\mathcal{I}$ and column indices $\mathcal{J}$ of cardinality $r$. For example, $\mathcal{I} = \mathcal{J} = \{1, ..., r \}$.}
\ENSURE{The updated sets $\mathcal{I}$ and $\mathcal{J}$ corresponding to the rank-$r$ dominant submatrix.
}
\STATE $C := A_{:,\mathcal{J}} A_{\mathcal{I},\mathcal{J}}^{-1}$
\STATE $current\_order := \{1, \ldots, M\}$
\STATE $C.swap(\mathcal{I}, \{1, \ldots, r\}, current\_order)$
\COMMENT{$C.swap(A,B,order)$ swaps the elements of $C$, corresponding to indexes from $A$ and $B$ ($A_i$ swaps with $B_i$) and changes the order of corresponding indexes in $order$.
}
\STATE $\{i, j\}: = \mathop {\arg \max }\limits_{i,j} \left| {{C_{i,j}}} \right|$
\STATE $changed := $ \TRUE
\STATE $old\_changed := $ \TRUE
\WHILE{$changed$}
  \STATE $changed := $ \FALSE
  \FOR{$changes\_in$ \textbf{in} $\{\mathcal{I}, \mathcal{J}\}$}
    \STATE{}
    \COMMENT{Here, we assign the pointer, so no copies of the sets are made.
    }
    \WHILE{$\left| C_{i,j} \right| > 1$}
      \IF{$i = j$}
        \STATE \textbf{break}
      \ENDIF
      \STATE $C := C - C_{:,j} \left( C_{i,:} - e_j \right)/C_{i,j}$
      \STATE $C.swap(i,j,current\_order)$
      \STATE $changed := $ \TRUE
      \STATE $old\_changed := $ \TRUE
    \ENDWHILE
    \IF{((\NOT $changed$) \AND ($changes\_in = \mathcal{J}$))}
      \STATE $changed := $ \FALSE
      \STATE \textbf{break}
    \ENDIF
    \IF{((\NOT $old\_changed$) \AND ($changes\_in = \mathcal{I}$))}
      \STATE $changed := $ \FALSE
      \STATE \textbf{break}
    \ENDIF
    \STATE $changes\_in := current\_order[1..r]$
    \COMMENT{The numbering starts at 1.}    
    \IF{$changes\_in = \mathcal{J}$}
      \STATE $old\_changed := $ \FALSE
      \STATE $C := A_{:,\mathcal{J}} A_{\mathcal{I},\mathcal{J}}^{-1}$    
      \STATE $current\_order := \{1, \ldots, M\}$
      \STATE $C.swap(\mathcal{I}, \{1, \ldots, r\}, current\_order)$
    \ELSE
      \STATE $C := A_{\mathcal{I},:}^T A_{\mathcal{I},\mathcal{J}}^{-T}$    
      \STATE $current\_order := \{1, \ldots, N\}$
      \STATE $C.swap(\mathcal{J}, \{1, \ldots, r\}, current\_order)$
    \ENDIF
    \STATE $\{i, j\} := \mathop {\arg \max }\limits_{i,j} \left| {{C_{i,j}}} \right|$
  \ENDFOR
\ENDWHILE
\end{algorithmic}

\paragraph*{Householder-based {\myfont maxvol2} algorithm\\}

Here we derive and present a faster algorithm for greedy row selection. The idea is the same as in the original {\myfont maxvol2} \cite{maxvol2}, but the complexity is lower.

Let some $k$ rows be selected. Matrix $A \in \mathC^{M \times r}$ can be represented in the form
\[
A = \left[ {\begin{array}{*{20}{c}}
  {\hat A} \\ 
  {\tilde A} 
\end{array}} \right],\quad \hat A \in {\mathbb{C}^{k \times r}}\, - \,{\text{current submatrix.}}
\]
For the fast update it is required to know the 2-norms of the rows from $\tilde A \hat A^+ \in \mathC^{(M-k) \times k}$. This matrix is of rank $r \leqslant k$, and therefore can be represented in the form
\[
\tilde A{{\hat A}^+ } = \tilde A M^{-1} Q,\quad Q^* M = \hat A,\quad Q \in {\mathbb{C}^{r \times k}},\quad Q{Q^ * } = I, \quad M \in \mathC^{r \times r}.
\]
For $k = r$, we can take $Q = I$. Next, we need to recalculate the matrix
\[
C = \tilde A M^{-1}  \in {\mathbb{C}^{M \times r}}.
\]
Only the permutations of the rows occur in the matrix $\tilde A$, so let's concentrate on the matrix $M$ recalculation. By appending the $i$-th row to the matrix $\hat A$, we obtain the matrix $\hat A' \in \mathC^{(k+1) \times r}$ of the form
\[
\hat A' = \left[ {\begin{array}{*{20}{c}}
  {{Q^ * }} \\ 
  {{C_{i,:}}} 
\end{array}} \right]M.
\]
Matrix $\left[ {\begin{array}{*{20}{c}}
  {{Q^ * }} \\ 
  {{C_{i,:}}} 
\end{array}} \right] \in \mathC^{(k+1) \times r}$ can be made unitary with the aid of the Householder reflector $H \mathC^{r \times r}$ such that
\[
{C_{i,:}}H = \left[ {\begin{array}{*{20}{c}}
  {{{\left\| {{C_{i,:}}} \right\|}_2}}&{\begin{array}{*{20}{c}}
  0& \ldots &0 
\end{array}} 
\end{array}} \right] = \left[ {\begin{array}{*{20}{c}}
  {\sqrt {{l_i}} }&{\begin{array}{*{20}{c}}
  0& \ldots &0 
\end{array}} 
\end{array}} \right]
\]
and through normalization of the first column by the matrix
\[D = diag(\frac{1}{{\sqrt {1 + {l_i}} }},1, \ldots ,1) \in \mathC^{r \times r}.
\]
Eventually
\[
\begin{gathered}
  \hat A' = \left[ {\begin{array}{*{20}{c}}
  {{Q^ * }} \\ 
  {{C_{i,:}}} 
\end{array}} \right]HD{D^{ - 1}}HM = Q'{D^{^{ - 1}}}HM = Q'M', \hfill \\
  {\left( {M'} \right)^{ - 1}} = {M^{ - 1}}HD, \hfill \\
  C' = \tilde A'{\left( {M'} \right)^{ - 1}} = \tilde A'{M^{ - 1}}HD, \hfill \\ 
\end{gathered} 
\]
that is, we only need to rearrange the columns, make the reflection and multiply the first column by a number. The change in the 2-norm of the rows occurs only at the last multiplication and can easily be calculated in terms of the 2-norm of the first column.

\begin{algorithm}[H]
\caption{Householder-based {\myfont maxvol2}}\label{hmaxvol2-alg}
\begin{algorithmic}[1]
\REQUIRE{Matrix $A \in \mathbb{C}^{M \times r}$, the starting set of row indices $\mathcal{I}$ of cardinality $r$, the required final size $n$.}
\ENSURE{$\mathcal{I}$, supplemented by $n-r$ row indices chosen greedily to maximize the volume.}
\STATE $C := A [A_{\mathcal{I},:}^{-1}, 0_{r \times (n-r)}]$
\STATE $current\_order := \{1, \ldots, M\}$
\STATE $C.swap(\mathcal{I}, \{1, \ldots, r\}, current\_order)$
\STATE $C := C_{\{r+1, \ldots, M\},:}$
\STATE $l := 0_{M-r}$
\FOR{$i := 1$ \TO $M-r$}
  \STATE $l_i := \| C_{i+r,:} \|_2^2$
\ENDFOR
\FOR{$new\_size := 1$ \TO $n-r$}
  \STATE $i := \mathop {\arg \max }\limits_i {l_i}$
  \STATE $current\_order.swap(new\_size+r,i+r)$
  \STATE $l_i' := 1 + l_i$
  \STATE $l_i := l_{new\_size}$
  \STATE $l_{new\_size} := 0$
  \STATE $C_I := C_{i,:}$
  \STATE $\tau, v := Householder(C_I)$
  \STATE $C_{i,:} := C_{new\_size,:}$
  \STATE $C_{new\_size,:} := 0_r$
  \STATE $C := C - \tau \left( C \left[ {\begin{array}{*{20}{c}}
  1 \\ 
  v 
\end{array}} \right] \right) \left[ {\begin{array}{*{20}{c}}
  1&{{v^ * }} 
\end{array}} \right]$
  \FOR{$j := new\_size$ \TO $M$}
    \STATE $l_j := l_j - \left( 1 - \frac{1}{l_i'} \right) |C_{j,1}'|^2$
  \ENDFOR
  \STATE $C_{:,1} := C_{:,1} / \sqrt{l_i'}$
\ENDFOR
\STATE $\mathcal{I} := current\_order[1..n]$
\end{algorithmic}
\end{algorithm}

\paragraph*{Algorithm {\myfont Dominant-C}}
\begin{algorithmic}[1]
\REQUIRE{Matrix $A \in \mathbb{C}^{M \times r}$, the starting set of row indices $\mathcal{I}$ of cardinality $n$. For example, $\mathcal{I} = \{1, ..., n \}$.}
\ENSURE{The updated set $\mathcal{I}$  corresponding to a dominant rectangular submatrix.}
\STATE $C := A A_{\mathcal{I},:}^+$
\STATE $current\_order := \{1, \ldots, M\}$
\STATE $C.swap(\mathcal{I}, \{1, \ldots, n\}, current\_order)$
\STATE $l := 0_M$
\FOR{$i := 1$ \TO $M$}
  \STATE $l_i := \| C_{i,:} \|_2^2$
\ENDFOR
\STATE $B := 0_{M \times n}$
\FOR{$i := n+1$ \TO $M$}
  \FOR{$j := 1$ \TO $n$}
    \STATE $B_{i,j} := |C_{i,j}|^2 + (1 + l_i) (1 - l_j)$
  \ENDFOR
\ENDFOR
\STATE \STATE $\{i,j\} := \mathop {\arg \max }\limits_{i,j} {B_{i,j}}$
\WHILE{$B_{i,j} > 1$}
  \STATE $C_I' := C_{i,:}^* / (1 + l_i)$
  \STATE $C' := C C_I'$
  \FOR{$k := 1$ \TO $M$}
    \STATE $l_k' := l_k - |C_k'|^2 (1 + l_i)$
  \ENDFOR
  \STATE $C := C - C' C_{i,:}$
  \STATE $C.swap(i,j,current\_order)$
  \STATE $C'.swap(i,j)$
  \STATE $l'.swap(i,j)$
  \STATE $swap(C_{:,j},C')$
  \FOR{$k := 1$ \TO $M$}
    \STATE $l_k := l_k' + |C_k'|^2 / (1 - l_i')$
  \ENDFOR
  \STATE $C := C + C' (C_{i,:} / (1 - l_i'))$
  \FOR{$i := n+1$ \TO $M$}
    \FOR{$j := 1$ \TO $n$}
      \STATE $B_{i,j} := |C_{i,j}|^2 + (1 + l_i) (1 - l_j)$
    \ENDFOR
  \ENDFOR
  \STATE $\{i,j\} := \mathop {\arg \max }\limits_{i,j} {B_{i,j}}$
\ENDWHILE
\STATE $\mathcal{I} := current\_order[1..n]$
\end{algorithmic}

\paragraph*{Algorithm {\myfont Dominant-R} (modified $RRQR$ from \cite{rrqr})\\}
As in \cite{rrqr}, the matrix sizes are denoted by $m$ and $n$. 
The approximation rank is $k$. In our case $n \gg m$, while \cite{rrqr} considers the case $m \geqslant n$. A small value of $n$ allows us not to worry about the update of $Q \in \mathbb{R}^{m \times r}$, and not to store it as a product of reflections. 

Some changes in the notation:

vector $\gamma$ contains \textbf{squared lengths} of the columns of $C^T$.

vector $\omega$ contains \textbf{squared lengths} of the rows of $A^{-1}$.

\begin{algorithmic}[1]
\REQUIRE{Matrix $M \in \mathbb{R}^{m \times n}$, the starting set of column indices $\mathcal{I}$ of cardinality $r$, threshold parameter $f \geqslant 1$. For example, $\mathcal{I} = \{1, ..., r \}$.}
\ENSURE{The updated set $\mathcal{I}$, corresponding to the dominant rectangular submatrix.}
\STATE $current\_order := \{1,...,n\}$
\STATE $M.swap\_columns(\mathcal{I}, \{1,...,k\}, current\_order)$ \COMMENT{the same as $swap$, but for columns.}
\STATE $Q, A := QR(M_{:,\mathcal{I}})$
\STATE $\gamma := 0_{n-k}$
\FOR{$i := 1$ \TO $n-k$}
  \STATE $\gamma_i := \| M_{:,i+k} \|_2^2$
\ENDFOR
\STATE $B := Q_{:,k}^T M$
\FOR{$i := 1$ \TO $n-k$}
  \STATE $\gamma_i := \gamma_i - \| B_{:,i+k} \|_2^2$
\ENDFOR
\STATE $\omega := 0_{k}$
\FOR{$i := 1$ \TO $k$}
  \STATE $\omega_i := \| A_{i,:}^{-1} \|_2^2$
\ENDFOR
\STATE $A^{-1}B := A^{-1} \cdot B$
\STATE $m := \omega^T \gamma$
\FOR{$i := 1$ \TO $k$}
  \FOR{$j := 1$ \TO $n-k$}
    \STATE $m_{i,j} := m_{i,j} + | (A^{-1}B)_{i+k,j+k}|^2$
  \ENDFOR
\ENDFOR
\STATE $\{i1,j1\} := \mathop {\arg \max }\limits_{i1,j1} {m_{i1,j1}}$
\STATE $\rho := m_{i1,j1}$
\item[]
\WHILE{$\rho > f$}
  \item[]
  \item[] \COMMENT{Exchange of $i1$ and $k$}
  \FOR{$i := i1$ \TO $k-1$}
    \STATE $\omega.swap(i,i+1)$
    \STATE $(A^{-1}B).swap\_columns(i,i+1)$
    \STATE $A.swap(i,i+1)$
    \STATE $M.swap\_columns(i,i+1,current\_order)$
  \ENDFOR
  \FOR{$i := i1$ \TO $k-1$}
    \STATE $\alpha, \beta := givens(A_{i,i}, A_{i+1,i})$ \COMMENT{see \cite{hqr}, page 240}
    \STATE $A_{i:i+1,i:k} := \left[ {\begin{array}{*{20}{c}}
  \alpha &-\beta \\ 
  \beta &\alpha  
\end{array}} \right] A_{i:i+1,i:k}$
    \STATE $Q_{:,i:i+1} := Q_{:,i:i+1} \left[ {\begin{array}{*{20}{c}}
  \alpha &\beta \\ 
  -\beta &\alpha  
\end{array}} \right]$
  \ENDFOR
  \item[]
  \item[] \COMMENT{Exchange $k, k+1, i1$ => $j1, k, k+1$}
  \STATE $b_c := M_{:,j1}$
  \STATE $M.swap\_columns(j1,k+1,current\_order)$
  \STATE $\gamma.swap(j1, k+1)$
  \STATE $(A^{-1}B).swap\_columns(j1, k+1)$
  \STATE $c_1^T := A_{k,k} (A^{-1}B)_{k,k+2:n}$ \label{Bre1}
  \STATE $A_{:,k} := A (A^{-1}B)_{:,k+1}$ \label{Bre2}
  \STATE $\gamma \nu := \sqrt{c_{k+1}}$ \label{Cre1}
  \STATE $\gamma := A_{k,k}$
  \STATE $b_1 := A_{1:k-1,k}$
  \STATE $q := (b_c - M_{:,1:k} (A^{-1}B)_{:,k+1})/(\gamma \nu)$ \label{Qre}
  \STATE $M.swap\_columns(k,k+1,current\_order)$
  \STATE $c_2^T := q^T M_{:,k+2:n}$ \label{Cre3}
  \item[]
  \item[] \COMMENT{Nullification of $k$-th column under the diagonal}
  \STATE $\alpha, \beta := givens(A_{k,k}, \gamma v)$
  \STATE $\left[ {\begin{array}{*{20}{c}}
  {\gamma \mu /\rho }&{\bar c_1^T} \\ 
  {\gamma \nu /\rho }&{\bar c_2^T} 
\end{array}} \right] := \left[ {\begin{array}{*{20}{c}}
  \alpha &-\beta \\ 
  \beta &\alpha  
\end{array}} \right] \left[ {\begin{array}{*{20}{c}}
  \gamma &{c_1^T} \\ 
  0&{c_2^T} 
\end{array}} \right]$
  \STATE $A_{k,k} := \sqrt{(\gamma \nu)^2 + A_{k,k}^2}$
  \item[]
  \item[] \COMMENT{Update of $\gamma$}
  \STATE $\gamma_{k+1} := (\gamma \nu / \rho)^2$
  \FOR{$i := k+2$ \TO $n$}
    \STATE $\gamma_i := \gamma_i + |\bar c_2^T{}_{i-k-1}|^2 - |c_2^T{}_{i-k-1}|^2$
  \ENDFOR
  \item[]
  \item[] \COMMENT{Update of $\omega$}
  \STATE $u := A_{k-1,k-1}.RTsolve(b_1)$ \COMMENT{Solving the system with the upper triangular matrix}
  \STATE $\omega_k := 1/A_{k,k}^2$
  \STATE $\mu := (A^{-1}B)_{k,k+1}$
  \STATE $\bar \mu := (A^{-1}B)_{1:rank-1,rank+1} + \mu u$
  \FOR{$i := 1$ \TO $k-1$}
    \STATE $\omega_i := \omega_i + \omega_k |\bar \mu_i|^2 - |u_i / \gamma|^2$
  \ENDFOR
  \item[]
  \item[] \COMMENT{Update of $A^{-1}B$}
  \STATE $u_1 := (A^{-1}B)_{1:k-1,k+1}$
  \STATE $(A^{-1}B)_{k,k+1} := (\gamma \mu / \rho) / A_{k,k}$
  \STATE $(A^{-1}B)_{k,k+2:n} := \bar c_1^T / A_{k,k}$
  \STATE $(A^{-1}B)_{1:k-1,k+1} := (1 - \mu (A^{-1}B)_{k,k+1})u - (A^{-1}B)_{k,k+1} u_1$
  \STATE $(A^{-1}B)_{1:k-1,k+2:n} := (A^{-1}B)_{1:k-1,k+2:n} + u(c_1^T/\gamma - \mu \bar c_1^T / A_{k,k}) - u_1 \bar c_1^T / A_{k,k}$
  \item[]
  \STATE $m := \omega^T \gamma$
  \FOR{$i := 1$ \TO $k$}
    \FOR{$j := 1$ \TO $n-k$}
      \STATE $m_{i,j} := m_{i,j} + | (A^{-1}B)_{i+k,j+k}|^2$
    \ENDFOR
  \ENDFOR
  \STATE $\{i1,j1\} := \mathop {\arg \max }\limits_{i1,j1} {m_{i1,j1}}$
  \STATE $\rho := m_{i1,j1}$
\ENDWHILE
\STATE $\mathcal{I} := current\_order[1..k]$
\end{algorithmic}

This is a modification of $GEQR$ \cite{rrqr} without recalculation of the matrices $Q \in \mathbb{R}^{m \times r}$, $B \in \mathbb{R}^{r \times (n-r)}$ and $C^T \in \mathbb{R}^{(m-r) \times (n-r)}$.
Thus in the complexity estimate of a single iteration, the coefficient of $mn$ is reduced by two times.

All notations and updates are the same as in the original article \cite{rrqr}. 
We focus only on the main changes: removing the updates of $B$ and $C$.

Getting rid of updating $B$ is quite simple: we need only the first row and the last column of $B$, which are easily computed through $A \in \mathbb{R}^{r \times r}$ and $A^{-1}B  \in \mathbb{R}^{r \times (n-r)}$ (lines \ref{Bre1} and \ref{Bre2}).

For the $C^T$, we only need to know the first row. 
Since unitary transformations were performed in $GEQR$ \cite{rrqr} for the matrix $C^T$, the first element of this row is equal to the 2-norm of the first column (line \ref{Cre1}).
The other elements are obtained as scalar products with the $k+1$-st column of $Q$ ($q$) (\ref{Cre3}).

To rotate, we still need to know $q = Q_{:,k+1} \in \mathbb{R}^m$ (line \ref{Qre}). 
It is easily obtained from the column $b_c \in \mathbb{R}^m$ of $M \in \mathbb{R}^{m \times n}$ as a product of $Q$ by the last column of the extension of~$A$:
\[{b_c} = \left[ {\begin{array}{*{20}{c}}
  {{Q_{:,1:k}}}&{{Q_{:,k + 1}}} 
\end{array}} \right]\left[ {\begin{array}{*{20}{c}}
  {A_{:,k}} \\ 
  {\gamma \nu } 
\end{array}} \right].\]
Then we use the equality
\[
Q_{:,1:k} A_{:,k} = M_{:,1:k} (A^{-1} B)_{:,{k+1}}
\]

Compared to $GEQR$ \cite{rrqr} we save one product of vectors of size $m-k$ and $n-k$ and one addition with $(m-k) \times (n-k)$ matrix. 
Update of $C$ also contains a matrix by vector multiplication. In {\myfont Dominant-R} we multiply the $m \times (n-k)$ matrix by a vector.

We can apply the same approach to append the columns. 
We do not repeat the derivation; the description of the original algorithm is in \cite{rrqr}. 
Here the benefit is even greater if we use $A^{-l} B$ later in {\myfont Dominant-R} or in {\myfont maxvol}. 

As before, $k = 1$ is not considered separately, although many operations disappear in this case. 
The maximum $k$ is denoted by $r$.

\paragraph*{Algorithm {\myfont pre-maxvol} (modification of the column addition in $RRQR$ from \cite{rrqr})}
\begin{algorithmic}[1]
\REQUIRE{Matrix $M \in \mathbb{R}^{m \times n}$, the required rank $r$.}
\ENSURE{Set of column indices $\mathcal{I}$ of cardinality $r$,  corresponding to the submatrix, whose volume differs from the maximum by no more than $r!$ times.}
\STATE $\mathcal{I} := \emptyset$
\STATE $A^{-1}B := 0_{r \times n}$
\STATE $\gamma := 0_n$
\FOR{$i := 1$ \TO $n$}
  \STATE $\gamma_i := \| M_{:,i} \|_2^2$
\ENDFOR
\STATE $j := \mathop {\arg \max }\limits_j {\gamma_j}$
\FOR{$k := 1$ \TO $r$}
  \STATE $\mathcal{I} := \mathcal{I} \cup \{j\}$
  \item[]
  \item[] \COMMENT{Exchange of $j$ and $k$}
  \STATE $b_c := M_{:,j}$
  \STATE $M.swap\_columns(j,k,current\_order)$
  \STATE $\gamma.swap(j, k)$
  \STATE $(A^{-1}B).swap\_columns(j, k)$
  \STATE $\gamma \nu := \sqrt{c_k}$
  \STATE $q := (b_c - M_{:,1:k-1} (A^{-1}B)_{1:k-1,k})/(\gamma \nu)$
  \STATE $c_2^T := q^T M_{:,k+1:n}$
  \item[]
  \item[] \COMMENT{Update of $\gamma$}
  \STATE $\gamma_k := 0$
  \FOR{$i := k+1$ \TO $n$}
    \STATE $\gamma_i := \gamma_i - \left|(c_2^T)_{i-k} \right|^2$
  \ENDFOR
  \item[]
  \item[] \COMMENT{Update of $A^{-1}B$}
  \STATE $(A^{-1}B)_{k,k+1:n} := c_2^T / A_{k,k}$
  \STATE $(A^{-1}B)_{1:k-1,k+1:n} := (A^{-1}B)_{1:k-1,k+1:n} - (A^{-1}B)_{1:k-1,k+1} (A^{-1}B)_{k,k+1:n}$
  \item[]
  \STATE $j := \mathop {\arg \max }\limits_j {\gamma_j}$
\ENDFOR
\end{algorithmic}

\section{Acknowledgments}
The work was supported by the Russian Science Foundation, Grant 14-11-00806.

\par\bigskip\noindent

\end{document}